\documentclass{article}
\usepackage[latin1]{inputenc}
\usepackage{amsmath}
\usepackage{amsthm,amssymb,amsfonts}
\usepackage{graphicx}
\usepackage{verbatim}
\usepackage{color,cite}
\usepackage{latexsym}
\usepackage{lscape}
\usepackage[T1]{fontenc}
\usepackage[all,cmtip]{xy}
\usepackage{caption}
\usepackage{xspace}
\usepackage{enumitem}
\usepackage{hyperref}
\usepackage{tikz}
\usepackage{ulem}
\usepackage{xcolor}
\usetikzlibrary{patterns}
\usetikzlibrary{shapes,arrows,spy,positioning}
\usetikzlibrary{decorations.markings}
\usepackage{caption}
\usepackage{subcaption}
\usepackage{float}
\usepackage{url}

\theoremstyle{plain}
\newtheorem{theorem}{Theorem}[section]
\newtheorem{lemma}[theorem]{Lemma}
\newtheorem{proposition}[theorem]{Proposition}

\newtheorem{sublemma}[theorem]{Sublemma}

\newtheorem{sublem}[theorem]{Sublemma}

\newtheorem*{question}{Question}

\tikzset{->-/.style={decoration={
  markings,
  mark=at position #1 with {\arrow{>}}},postaction={decorate}}}
\tikzset{every loop/.style={}}
\tikzset{circ/.style={shape=circle, inner sep=2pt, draw}}

\newcommand{\genplin}{\text{GL}^+_2(\mathbb{R})}

\newcommand{\nc}{\newcommand}
\newcommand{\rnc}{\renewcommand}

\nc\bB{\mathbb{B}}
\nc\bC{\mathbb{C}}
\nc\bD{\mathbb{D}}
\nc\bE{\mathbb{E}}
\nc\bF{\mathbb{F}}
\nc\bG{\mathbb{G}}
\nc\bH{\mathbb{H}}
\nc\bI{\mathbb{I}}
\nc{\bJ}{\mathbb{J}}
\nc\bK{\mathbb{K}}
\nc\bL{\mathbb{L}}
\nc\bM{\mathbb{M}}
\nc\bN{\mathbb{N}}
\nc\bO{\mathbb{O}}
\nc\bP{\mathbb{P}}
\nc\bQ{\mathbb{Q}}
\nc\bR{\mathbb{R}}
\nc\bS{\mathbb{S}}
\nc\bT{\mathbb{T}}
\nc\bU{\mathbb{U}}
\nc\bV{\mathbb{V}}
\nc\bW{\mathbb{W}}
\nc\bY{\mathbb{Y}}
\nc\bX{\mathbb{X}}
\nc\bZ{\mathbb{Z}}

\nc\cA{\mathcal{A}}
\nc\cB{\mathcal{B}}
\nc\cC{\mathcal{C}}
\nc\cD{\mathcal{D}}
\nc\cE{\mathcal{E}}
\nc\cF{\mathcal{F}}
\nc\cG{\mathcal{G}}
\nc\cH{\mathcal{H}}
\nc\cI{\mathcal{I}}
\nc{\cJ}{\mathcal{J}}
\nc\cK{\mathcal{K}}
\nc\cL{\mathcal{L}}
\nc\cM{\mathcal{M}}
\nc\cN{\mathcal{N}}
\nc\cO{\mathcal{O}}
\nc\cP{\mathcal{P}}
\nc\cQ{\mathcal{Q}}
\nc\cS{\mathcal{S}}
\nc\cT{\mathcal{T}}
\nc\cU{\mathcal{U}}
\nc\cV{\mathcal{V}}
\nc\cW{\mathcal{W}}
\nc\cY{\mathcal{Y}}
\nc\cX{\mathcal{X}}
\nc\cZ{\mathcal{Z}}
\nc\ra{\rightarrow}

\nc\fC{\mathfrak{C}}
\nc\bfA{\mathbf{A}}
\nc\bfB{\mathbf{B}}
\nc\bfC{\mathbf{C}}
\nc\bfD{\mathbf{D}}
\rnc{\Col}{\mathrm{Col}}
\nc\For{\mathcal{F}}
\nc{\rank}{\mathrm{rank}}

\title{Algebraically primitive invariant subvarieties with quadratic field of definition}
\author{Paul Apisa and David Aulicino}
\date{}

\begin{document}

\maketitle

\begin{abstract}
We show that the only algebraically primitive invariant subvarieties of strata of translation surfaces with quadratic field of definition are the Teichm\"uller curve generated by the decagon, Weierstrass curves and eigenform loci in genus two, and the rank two example in the minimal stratum of genus four translation surfaces discovered by Eskin-McMullen-Mukamel-Wright.
\end{abstract}


\section{Introduction}

The $\mathrm{GL}(2, \bR)$ action on strata of translation surfaces  has  orbit closures that are linear orbifolds, called \emph{invariant subvarieties}, by work of Eskin-Mirzakhani \cite{EskinMirzakhaniInvariantMeas} and Eskin-Mirzakhani-Mohammadi \cite{EskinMirzakhaniMohammadiOrbitClosures}. These orbifolds are indeed varieties by work of Filip \cite{Filip2}. They are also defined by linear equations in period coordinates on strata. Given an invariant subvariety $\cM$, the smallest field $k(\cM)$ in which the coefficients of these linear equations may be taken is the \emph{field of definition}. If $(X, \omega)$ is a translation surface in $\cM$, the tangent space of $\cM$ at $(X, \omega)$ is naturally identified with a subspace of $H^1(X, Z; \bC)$, where $Z$ is the set of cone points of $(X, \omega)$. The image of $T_{(X, \omega)} \cM$ in $H^1(X; \bC)$ under the map $p: H^1(X, Z; \bC) \ra H^1(X; \bC)$ that sends relative cohomology to absolute cohomology determines a subspace $V := p(T_{(X, \omega)} \cM)$. By work of Avila-Eskin-M\"oller \cite{AvilaEskinMollerForniBundle}, $V$ is symplectic. The \emph{rank} of $\cM$ is half the dimension of $V$ and the \emph{rel} of $\cM$ is $\dim \ker p \vert_{T_{(X, \omega)} \cM}$. The rank, rel, and field of definition were originally defined by Wright \cite{WrightCylDef, WrightFieldofDef} who showed that $k(\cM)$ is a totally real finite extension of $\bQ$ and that $V$ has a basis in $H^1(X; k(\cM))$. This implies that $V^\sigma$ is well-defined for each field embedding $\sigma: k(\cM) \ra \bR$. Wright showed that $\bigoplus_\sigma V^\sigma \subseteq H^1(X)$, where the direct sum is taken over all field embeddings of $k(\cM)$ in $\bR$. 
When equality holds, $\cM$ is called \emph{algebraically primitive}. An equivalent formulation is that $[k(\cM): \bQ] \cdot \mathrm{rank}(\cM) = g$, where $g$ is the genus of the surfaces in $\cM$. \footnote{When $\cM$ is the orbit of a Veech surface, $\mathrm{rank}(\cM) = 1$ and $k(\cM)$ is the trace field. So, for Veech surfaces, an equivalent definition of algebraic primitivity is that the trace field has degree $g$ over $\bQ$.  This definition coincides with the one given in M\"oller \cite{Moller-PeriodicPoints}.}

Calta \cite{CaltaVeechSurfsCompPerGen2} and McMullen \cite{McMullenBilliardsAndTeichCurves} discovered infinitely many algebraically primitive invariant subvarieties in genus two. McMullen \cite{McMullenDiscriminantSpin, McMullenDecagonUnique, McMullenGenus2} then classified all algebraically primitive orbit closures in genus two. A description of these orbit closures can be found in McMullen \cite{McMullen-BilliardsTeichCurveSurvey}. More recently, Eskin-McMullen-Mukamel-Wright \cite{EskinMcMullenMukamelWright} discovered an algebraically primitive invariant subvariety in the minimal stratum in genus four, which we call the \emph{$(1,1,1,7)$-locus}. The name is apt since this orbit closure can be described as the smallest invariant subvariety containing all unfoldings of $(1,1,1,7)$ quadrilaterals. Our main result is the following. 


\begin{theorem}\label{T:main}
If $\cM$ is an algebraically primitive invariant subvariety with quadratic field of definition, then, up to forgetting marked points, $\cM$ is one of the following:  the Teichm\"uller curve generated by the regular decagon, an eigenform locus of genus two surfaces in $\cH(1,1)$, a Weierstrass curve in $\cH(2)$, or the $(1,1,1,7)$-locus. 
\end{theorem}

This result should be compared to Mirzakhani-Wright \cite{MirzakhaniWrightFullRank} which classified \emph{full rank} invariant subvarieties, i.e., algebraically primitive invariant subvarieties with rational field of definition, as components of strata and hyperelliptic loci up to ignoring marked points. Up to ignoring marked points, the only known algebraically primitive orbit closures aside from full rank ones and those mentioned in Theorem \ref{T:main} are:
\begin{itemize}
    \item Veech's double regular $p$-gon surface, where $p$ is prime,
    \item Veech's regular $m$-gon surface, where $m$ is twice a prime or a power of two, and
    \item two sporadic Teichm\"uller curves in genus three and four found by Kenyon-Smillie \cite{KenyonSmillie} and Vorobets \cite{Vorobets-Sporadic}.
\end{itemize}
  See Wright \cite[Theorem 1.7]{WrightSchwarzVWBMCurves} for a proof that this list includes all of the algebraically primitive Bouw-M\"oller curves; see Veech \cite{VeechTeichCurvEisen} for a construction of the first two examples; see Leininger \cite{Leininger-Sporadic} and McMullen \cite[Section 6]{McMullen-BilliardsTeichCurveSurvey} for a construction of the sporadic Teichm\"uller curves.

Theorem \ref{T:main} may be contextualized in light of the main theorem of Apisa-Wright \cite{ApisaWrightHighRank}, which states that any invariant subvariety of genus $g$ surfaces whose rank is at least $\frac{g}{2}+1$ is a component of a stratum or a full locus of holonomy double covers of a component of a stratum of quadratic differentials. Given this result, the highest rank invariant subvarieties that one could hope to find, when $g$ is even, excluding loci of covers, would have rank $\frac{g}{2}$.  Such subvarieties could either have rational or quadratic field of definition. Theorem \ref{T:main} classifies the examples in the latter case. Moreover, the theorem provides some evidence for a negative resolution of the question posed in Apisa-Wright \cite[Question~1.3]{ApisaWrightHighRank}, as to whether there exists a rank three invariant subvariety which does not arise as a locus of branched covers of smaller genus surfaces. 

The proof of Theorem \ref{T:main} relies heavily on the recent work of Apisa and Wright \cite{ApisaWrightDiamonds, ApisaWrightGeminal, ApisaWrightHighRank}. We anticipate that these techniques will continue to be useful for classifying invariant subvarieties. In contrast with the original application of these techniques in \cite{ApisaWrightHighRank}, the present work applies these techniques to study nonarithmetic invariant subvarieties.

\vspace{.25cm}

\noindent \textbf{Outline of the Proof of Theorem \ref{T:main} and Future Directions.} \text{}

\vspace{.25cm}

\noindent After proving general results about the boundary of algebraically primitive invariant subvarieties in Section \ref{S:Boundary}, we will explain in Section \ref{S:ProofMainTheorem} how to deduce Theorem \ref{T:main} from the established classification results in genus two and the following two special cases:

\begin{theorem}\label{T:main-tool}
The only algebraically primitive rank two rel zero invariant subvariety with quadratic field of definition is the $(1,1,1,7)$-locus.  In particular, the $(1,1,1,7)$-locus has no periodic points. 
\end{theorem}

\begin{theorem}\label{T:Main-tool2}
The only algebraically primitive rank two rel one invariant subvariety with quadratic field of definition is the $(1,1,1,7)$-locus with a free marked point.
\end{theorem}

\noindent These results are shown in Sections \ref{S:Rk2Rel0} and \ref{S:Rk2Rel1}, respectively. Section \ref{S:CylinderRigidInGenusTwo} classifies cylinder rigid rank one rel one algebraically primitive subvarieties in genus two, which is crucial for proving Theorem \ref{T:main-tool}.

In principle, the only obstacle to classifying all algebraically primitive invariant subvarieties of rank at least two with the techniques developed in this paper is the following.

\begin{question}
Is it possible to classify all cylinder rigid rank one rel one algebraically primitive subvarieties? As a particularly interesting special case, can these be classified in genus three?
\end{question}

\noindent A resolution to the special case could lead to a classification of algebraically primitive invariant subvarieties of rank at least two with cubic field of definition. By work of Apisa \cite{Apisa-RankOneHyp} and Ygouf \cite{Ygouf-RankOneGenusThree}, ignoring marked points, cylinder rigid rank one rel one algebraically primitive subvarieties in genus three only exist in the principal stratum or the stratum with exactly three zeros.

\subsubsection*{Connection to previous work.} Algebraically primitive Teichm\"uller curves in genus three were studied extensively by Bainbridge-M\"oller \cite{BainbridgeMoller} and Bainbridge-Habegger-M\"oller \cite{BainbridgeHabeggerMollerHNFilt}, where an effective finiteness result was shown. They were classified in the hyperelliptic component of $\cH(2,2)$ by Winsor \cite{Winsor-H22}. Finiteness of algebraically primitive orbit closures was established for genus greater than two by Eskin-Filip-Wright \cite{EskinFilipWrightAlgHull} and in certain special cases in earlier work of Matheus-Wright \cite{MatheusWrightHodgeTeichPlanes}, Lanneau-Nguyen-Wright \cite{LanneauNguyenWrightFinInNonArithRkOne}, and M\"oller \cite{Moller-TeichCurvesInHyp}.  Filip \cite{Filip-Survy} described a connection between the notion of algebraic primitivity and other natural invariants of invariant subvarieties. Finally, characteristic numbers of the algebraic surfaces formed by the restriction of the tautological curve to an algebraically primitive Teichm\"uller curve were determined for all known examples by Freedman-Lucas \cite{Freedman-Lucas}.

\subsubsection*{Acknowledgments.} The authors are grateful to Alex Wright for helpful conversations.  The authors are grateful to the anonymous referee for their careful reading and detailed feedback.  P.A. was partially supported by NSF Grant DMS Award No. 2304840.  D.A. was partially supported by the National Science Foundation under Award Nos. DMS - 1738381, DMS - 1600360, Simons Foundation: Collaboration Grant under Award No. 853471, and several PSC-CUNY Grants.

\subsubsection*{AI Disclosure.} No AI was used in the first version of this paper.  Upon receiving the referee report, a gap was found in the proofs of Sublemma 6.4 and Theorem 1.3.  There were two surfaces that we needed to prove did not generate rank two rel one orbit closures.  Originally, we had Claude Opus 5.5 run compute their orbit closures in flatsurf. Upon seeing the output, the authors found simpler proofs, one of which appears in Sublemma 6.4 and the other, which is still computational, which is described in a Jupyter notebook, written by Claude, that appears on the second author's website. ChatGPT was also used to check the paper and the authors addressed the issues it found.
%

\section{The boundary of algebraically primitive invariant subvarieties}\label{S:Boundary}

The results below heavily use the techniques developed in Apisa-Wright \cite{ApisaWrightDiamonds, ApisaWrightGeminal, ApisaWrightHighRank}. We will also import the notation from those papers. We will briefly recapitulate the main definitions and notation here. We denote by $(X, \omega)$ a \emph{translation surface} consisting of a Riemann surface $X$ together with a nonzero holomorphic $1$-form $\omega$. In the sequel, $\cM$ will always be an invariant subvariety. 

Let $(X, \omega) \in \cM$. An \emph{equivalence class of cylinders on $(X, \omega)$} is a maximal collection of parallel cylinders that remain parallel to one another at all surfaces in a small neighborhood of $(X, \omega)$ in $\cM$. An equivalence class $\bfC$ is called \emph{generic} if the boundary saddle connections of the cylinders in $\bfC$ remain parallel to the cylinders in $\bfC$ on all nearby surfaces in $\cM$ and if, additionally, no nearby surface in $\cM$ contains cylinders not belonging to $\bfC$ and that are equivalent to cylinders in $\bfC$.

The space of all linear combinations of Poincar\'e-duals of core curves of cylinders in $\bfC$ that belong to $T_{(X, \omega)}\cM$ is called the \emph{twist space $\mathrm{Twist}(\bfC)$ of $\bfC$}. The cylinder deformation theorem of Wright \cite{WrightCylDef} states that the \emph{standard shear $\sigma_{\bfC}$}, which is the linear combination of cylinder core curves whose coefficients are the heights of the corresponding cylinder, belongs to the twist space. The equivalence class $\bfC$ is said to be \emph{involved with rel} if $\bfC$ contains a cylinder with a cross curve on which a rel vector, i.e., an element of $\ker p|_{T_{(X, \omega)} \cM}$, evaluates nontrivially. Recall that a \emph{cross curve of a cylinder $C$} is a saddle connection that is contained in $\overline{C}$ and that intersects its core curve exactly once. 

Given $v \in \mathrm{Twist}(\bfC)$, the cylinders $\bfC_v$ are those in $\bfC$ whose heights vanish first along the path $(X, \omega) + tv$ parameterized by $t \geq 0$.  We will always suppose that $v$ is \emph{suitable}, that is, along this path, the length of some saddle connection contained in $\overline{\bfC}$ goes to zero.  We say that $v$ is \emph{typical} if $\bfC$ is generic and the cylinders in $\bfC_v$ have constant ratios of heights for all perturbations of $(X, \omega)$ in $\cM$. Finally, given a suitable element of the twist space $v$, we can use it to deform the surface $(X, \omega)$ to a boundary surface $\Col_v(X,\omega)$, which belongs to a boundary invariant subvariety $\cM_v$ of $\cM$, by collapsing a saddle connection. If $\bfC$ makes an angle of $\theta$ with the horizontal and $v = -e^{i\theta}\sigma_{\bfC}$, then we simply write $\cM_{\bfC}$ (see \cite[Section 1.1]{ApisaWrightDiamonds} for a discussion of this special case).  By Apisa-Wright \cite[Lemma~11.2]{ApisaWrightHighRank}, if $\bfC$ is generic, then it is always possible to find a typical element of $\mathrm{Twist}(\bfC)$, which implies that $\dim(\cM_v)= \dim(\cM) - 1$ by \cite[Lemma 6.5]{ApisaWrightHighRank}.  Moreover, if $\bfC$ is involved with rel, then $\mathrm{rank}(\cM_v) = \mathrm{rank}(\cM)$ by \cite[Lemma 11.4]{ApisaWrightHighRank}.

\begin{lemma}\label{L:GeneralInvolvedWithRel}
Suppose that $\bfC$ is a generic equivalence class of cylinders on a surface $(X, \omega)$ in an invariant subvariety $\cM$. Then $\bfC$ is not involved with rel if and only if every saddle connection contained in $\overline{\bfC}$ and that is not contained in the boundary of a cylinder in $\bfC$ pairs trivially with every rel vector in $T_{(X, \omega)} \cM$. 
\end{lemma}
\begin{proof}
For the reverse implication, if $\bfC$ is involved with rel, then, by definition, some cross curve of a cylinder in $\bfC$ pairs nontrivially with a rel vector in the tangent space of $\cM$. 

For the forward implication, suppose that $\bfC$ is not involved with rel, i.e. every cylinder cross curve of $\bfC$ pairs trivially with every element of $\ker p|_{T_{(X, \omega)} \cM}$. Let $s$ be a saddle connection contained in $\overline{\bfC}$ and that is not contained in the boundary of a cylinder in $\bfC$. We must show that $s$ also pairs trivially with every rel vector in the tangent space of $\cM$. Suppose without loss of generality that $\bfC$ is horizontal. Let $t$ be a cross curve of a cylinder in $\bfC$ and $\gamma$ its core curve. Let $a \in \bR \setminus \{0\}$ be such that $at$ and $s$ have the same imaginary parts. Let $b \in \bR$ be such that $\int_{s-at-b\gamma}\omega = 0$. Since $\bfC$ is not involved with rel, $\mathrm{Twist}(\bfC)$ is one-dimensional (otherwise it contains a rel vector by Mirzakhani-Wright \cite[Theorem 1.5]{MirzakhaniWrightBoundary}). Together with the twist space decomposition \cite[Proposition 4.20]{ApisaWrightHighRank}, this implies that every surface near $(X, \omega)$ in $\cM$ can be reached by a composition of three operations: applying an element of $\mathrm{GL}(2, \bR)$ to the surface, applying a piecewise linear map to the surface that acts as the identity on $\bfC$, and applying a piecewise linear map to the surface acts by the identity on the complement of $\bfC$ and by an element of $\begin{pmatrix} 1 & a \\ 0 & b \end{pmatrix}$ for some $a, b \in \bR$ on $\bfC$. Each of these operations preserves the fact that $s-at-b\gamma$ has vanishing period. So $s-at-b\gamma$ evaluates to zero on $T_{(X, \omega)} \cM$. If $v \in T_{(X, \omega)} \cM$ is a rel vector, then $v(\gamma) = 0$, and so $v(s) = av(t) = 0$, as desired.
\end{proof}

\begin{proposition}\label{CylinderSpan:Prop}
Suppose that $\bfC$ is a generic equivalence class of cylinders on a surface $(X, \omega)$ in a degree $d$ algebraically primitive invariant subvariety $\cM$. Suppose too that $\bfC$ is not involved with rel and that its closure is not $(X, \omega)$. Then $\Col_{\bfC}(X, \omega)$ has genus at least $d$ less than the genus of $(X, \omega)$. 
\end{proposition}
\begin{proof}
Let $w$ be any saddle connection in $\overline{\bfC}$ that vanishes under the collapse. Since $\bfC$ is generic and not involved with rel, $w$ pairs trivially with elements of $\ker(p) \cap T_{(X, \omega)} \cM$ by Lemma \ref{L:GeneralInvolvedWithRel}.  The rank of $\cM_{\bfC}$ is strictly less than that of $\cM$ by Apisa-Wright \cite[Lemma 3.8]{ApisaWrightHighRank}. By Apisa-Wright \cite[Theorem 5.1]{ApisaWrightHighRank}, $\Col_{\bfC}(\bfC)$ is rel-scalable. Therefore, no core curve of a cylinder in $\bfC$ can represent an absolute homology class after degenerating, and the genus of $\Col_\bfC(X, \omega)$ must be less than that of $(X, \omega)$. In particular, if $\cH$ is the stratum containing $(X, \omega)$, then the rank of $\cH_\bfC$ is less than that of $\cH$. By Apisa-Wright \cite[Lemma 3.8]{ApisaWrightHighRank}, it follows that there is a nonzero vanishing cycle $v$ that vanishes when paired with any element of $\ker(p)$, i.e., $v$ is an absolute cycle\footnote{Even though $w$ paired trivially with rel vectors in $\cM$, $w$ is not necessarily an absolute cycle. The bootstrapping argument given here nonetheless finds a vanishing cycle that is an absolute cycle.}. Set $V_\sigma := p(T\cM)^\sigma$ and $V' := \bigoplus_\sigma V_\sigma$ with the sum ranging over embeddings $\sigma$ of the totally real field $k(\cM)$ into $\bR$. By Wright \cite{WrightFieldofDef}, this direct sum is of symplectic and symplectically orthogonal subspaces and, since $\cM$ is algebraically primitive,  $H^1(X; \bC) = V'$. By Poincar\'e duality, we will abuse notation and say that this is also a direct sum decomposition of $H_1(X; \bC)$. By Filip \cite{Filip2}, there is an order $\cO$ in $k(\cM)$ that acts on $H_1(X; \bZ)$ by real multiplication where $\alpha \in \cO$ acts on $V_\sigma$ by multiplication by $\sigma(\alpha)$. The subspace spanned by the absolute vanishing cycles is invariant under the action of $\cO$ by Bainbridge-M\"oller \cite[Proposition 5.5]{BainbridgeMoller}. Write $v = \sum_\sigma v_\sigma$ where $v_\sigma \in V_\sigma$. Let $\alpha$ be a primitive element in $\cO$. Then $\alpha^k \cdot v = \sum_\sigma \sigma(\alpha^k) v_\sigma$ for any integer $k$. Labeling the embeddings of $k(\cM)$ into $\bR$ as $\sigma_1,\hdots, \sigma_d$ and using that the $d \times d$ matrix $(\sigma_i(\alpha^j))$ is invertible, it follows that the span of $\cO \cdot v$ has a basis $(v_{\sigma_i})_{i=1}^d$ and, since $v$ is integral, $v_{\sigma_i} \ne 0$ for all $i$. Since these basis elements are pairwise symplectically orthogonal, $\cO \cdot v$ spans an isotropic $d$-dimensional subspace. By Chen-Wright \cite{ChenWrightWYSIWYG}, recalling that $\cH$ is the stratum, $T_{\Col_{\bfC}(X, \omega)} \cH_{\bfC}$ can be identified with $\mathrm{Ann}(V)$, which is the annihilator of the space $V$ of vanishing cycles in $H_1(X, \Sigma; \bC)$, where $\Sigma$ is the set of singularities of $(X, \omega)$. The genus of $\Col_{\bfC}(X, \omega)$ is half the dimension of $\cH_{\bfC}$, which is the rank of the maximal symplectic subspace of $p(\mathrm{Ann}(V))$ by Apisa-Wright \cite[Corollary 3.5]{ApisaWrightHighRank}. This subspace is contained in $p(\mathrm{Ann}(\cO \cdot v))$, and hence has dimension at most $2g-2d$. Thus, $\Col_{\bfC}(X, \omega)$ has genus at least $d$ less than that of $(X, \omega)$.
\end{proof}

\begin{lemma}\label{lemma:genusbytwo}
Suppose that $\bfC$ is a generic equivalence class of cylinders on a surface $(X, \omega)$ in an algebraically primitive invariant subvariety $\cM$. If $v \in \mathrm{Twist}(\bfC)$ is typical, then $\cM_v$ is algebraically primitive. 

If $g$, $d$, and $r$ are the genus, degree, and rank of $\cM$, respectively, then the (possibly disconnected) surfaces in $\cM_v$ have genus $g$ if $\bfC$ is involved with rel and $g-d$ otherwise. 
\end{lemma}
\begin{proof}
Since the field of definition is preserved under degenerations\footnote{This is shown in Mirzakhani-Wright \cite[Corollary 2.14]{MirzakhaniWrightBoundary} using a consequence of \cite[Theorem 2.9]{MirzakhaniWrightBoundary}, that was conditional in the case that $\cM_v$ is disconnected. The result was shown unconditionally in Chen-Wright \cite[Lemma 6.5]{ChenWrightWYSIWYG}. }, the degree of $\cM_v$ is also $d$. If $\bfC$ is involved with rel, then $\mathrm{rank}(\cM_v) = \mathrm{rank}(\cM)$ as explained above. Since $g = dr$ by assumption and since the product of degree and rank is always a lower bound on genus, it follows that the surfaces in $\cM_v$ are still genus $g$ surfaces. 

Therefore, it remains to consider the case where $\bfC$ is not involved with rel. In this case, $\mathrm{rank}(\cM_v) = r -1$ as explained above and so the product of degree and rank, which is $g-d$ is a lower bound on the genus of surfaces in $\cM_v$. 
It is also an upper bound by Proposition \ref{CylinderSpan:Prop}. 
\end{proof}

\begin{lemma}
\label{lemma:ConnCompsDeg}
If $d$, $v$, $g$, and $\cM$ are as in Lemma \ref{lemma:genusbytwo}, then $\cM_v$ is prime and the surfaces in $\cM_v$ have at most $d$ components provided that $\cM$ has rank at least two. If equality holds and $\mathrm{rank}(\cM_v) > 1$, then $\cM_v$ is contained in $\cM_1 \times \hdots \times \cM_d$, where $\cM_i$ is a hyperelliptic locus or connected component of a stratum of surfaces of genus $\rank(\cM_v)$.
\end{lemma}
\begin{proof}
Assume that a surface in $\cM_v$ has components labeled $\{1, \hdots, n\}$ and that the $i$th component has genus $g_i$. Let $\cM_i$ be the projection of $\cM_v$ onto the $i^{th}$ component.  Since, by Apisa-Wright \cite[Lemma 9.1]{ApisaWrightHighRank}, $\cM_v$ is prime, each $\cM_i$ has rank equal to that of $\cM_v$ by \cite[Theorem~1.3(3)]{ChenWrightWYSIWYG}.  Furthermore, $g_i \geq \rank(\cM_v)$ for all $i$.

First assume that $\rank(\cM_v) = \rank(\cM)$, then 
\[ \rank(\cM) = \frac{g}{d} \geq \frac{1}{d} \sum g_i \geq \frac{n}{d}\rank(\cM). \]
It is immediate that $d \geq n$ and, if equality holds, then $g_i = \mathrm{rank}(\cM_v)$ for all $i$.

Now assume that $\rank(\cM_v) = \rank(\cM)-1$. By Lemma~\ref{lemma:genusbytwo}, $g-d \geq \sum g_i$.  This yields
\[ \rank(\cM) = \frac{g}{d} \geq \frac{1}{d} \left( d + \sum g_i \right) \geq 1 + \frac{n}{d} \rank(\cM_v). \]
Equivalently,
\[ \rank(\cM) - 1 \geq \frac{n}{d} \left( \rank(\cM) - 1 \right). \]
Since $\rank(\cM) > 1$, it follows that $d \geq n$ and, if equality holds, then $g_i = \mathrm{rank}(\cM_v)$ for all $i$.

When $\cM_v$ consists of surfaces with exactly $d$ connected components, we have seen that $\cM_i$ has full rank (and rank at least two) for all $i$ by assumption and by primality. The final claim holds by Mirzakhani-Wright \cite[Theorem 1.1]{MirzakhaniWrightFullRank}.
\end{proof}

\section{Proof of Theorem \ref{T:main} assuming Theorems \ref{T:main-tool} and \ref{T:Main-tool2}}\label{S:ProofMainTheorem}

\begin{lemma}
\label{lemma:NonArithConnProdTwoSC}
If $\cM$ is a nonarithmetic invariant subvariety and $\textbf{C}$ is a generic equivalence class of cylinders on a surface $(X, \omega) \in \cM$ that does not cover $(X, \omega)$ and whose twist space is spanned by $\sigma_{\bfC}$, then $\Col_{\bfC}(\bfC)$ contains at least two saddle connections.
\end{lemma}

\begin{proof}
Suppose in order to derive a contradiction that $\Col_\bfC(\bfC)$ consists of a single saddle connection. Suppose without loss of generality that $\bfC$ is horizontal. By Apisa-Wright \cite[Lemma 4.13]{ApisaWrightHighRank}, excluding vertical saddle connections, every vertical leaf through a point in $\bfC$ must eventually intersect the boundary of $\overline{\bfC}$. Since $\Col_{\bfC}(\bfC)$ is a single saddle connection, it follows that there are two saddle connections $\sigma_0$ and $\sigma_1$ on $(X, \omega)$ so that $\overline{\bfC}$ is one component of $(X, \omega) - (\sigma_0 \cup \sigma_1)$. Without loss of generality, let $\sigma_0$ (resp. $\sigma_1$) be the bottom (resp. top) boundary of $\overline{\bfC}$.  

All vertical trajectories traveling up from $\sigma_0$ must pass through $\sigma_1$. Let $I$ be the interior of $\sigma_0$. As noted in the previous paragraph, every maximal vertical line segment contained in $\overline{\bfC}$, aside from vertical saddle connections, begins at an interior point of $\sigma_0$ and ends at an interior point of $\sigma_1$. Therefore, after deleting finitely many points (coming from intersections with vertical saddle connections), every horizontal saddle connection in $\overline{\bfC}$ can be covered by finitely many copies of $I$ that have been translated by the vertical straight-line flow. So every horizontal saddle connection in $\overline{\bfC}$ has a length that is an integral multiple of the length of $I$. We conclude that every cylinder in $\bfC$ has circumference which is an integral multiple of the length of $I$.  This contradicts the fact that there exist two cylinders in $\bfC$ that have an irrational ratio of circumferences since $\cM$ is nonarithmetic (by Wright \cite[Theorem 7.1]{WrightCylDef}).
\end{proof}

The following useful lemma is essentially contained in Apisa-Wright \cite{ApisaWrightHighRank}.

\begin{lemma}\label{L:StillGeneric}
Let $\cM$ be an invariant subvariety and let $\bfC$ be an equivalence class of generic cylinders with a typical degeneration $v$. Then the cylinders in $\Col_v(\bfC)$ are generic when this set is nonempty \color{red}. 
\end{lemma}
\begin{proof}
It suffices to show that every saddle connection on the boundary of a cylinder in $\Col_v(\bfC)$ is generically parallel to the cylinders that remain in $\Col_v(\bfC)$. Each such saddle connection is either in $\Col_v(\bfC_v)$ (if it was on the boundary of a cylinder in $\bfC_v$) or, otherwise, it can be written as $\Col_v(s)$ where $s$ is a saddle connection on the boundary of a cylinder in $\bfC$. In the latter case there is nothing to show. In the former, we are done by (the proof of) Lemma 6.5 in \cite{ApisaWrightHighRank}. 
\end{proof}

\begin{proposition}\label{P:induction}
Suppose that $\cM$ is an algebraically primitive invariant subvariety with $\rank(\cM) \geq 2$, quadratic field of definition, and whose surfaces contain no marked points. If $\rank(\cM) = 2$, then suppose that $\mathrm{rel}(\cM) > 0$.

Then there is an algebraically primitive component $\cM'$ of the boundary of $\cM$ so that $\dim(\cM') = \dim(\cM)-1$ and so $\cM'$ consists of connected surfaces with no free marked points. If $\mathrm{rel}(\cM) \geq 1$, then additionally $\rank(\cM') = \rank(\cM)$.

Moreover, if $\bfC_1$ and $\bfC_2$ are two disjoint generic equivalence classes of cylinders on a surface in $\cM$ where at least one is involved with rel if $\mathrm{rank}(\cM) = 2$, then there is a typical degeneration $v$ in one of their twist spaces so that $\cM' = \cM_v$.
\end{proposition}

\begin{proof}
Let $\bfC_1$ and $\bfC_2$ be two disjoint generic equivalence classes of cylinders on a surface $(X, \omega) \in \cM$ with disjoint boundaries. Suppose without loss of generality that $\bfC_1$ is involved with rel if $\mathrm{rel}(\cM) \geq 1$. This can be arranged by perturbing a cylindrically stable surface created in \cite[Lemma 7.10]{ApisaWrightHighRank}. (Note that equivalence classes can always be perturbed to become generic and this perturbation preserves disjointness of equivalence classes.)

Let $v$ be a typical degeneration in $\mathrm{Twist}(\bfC_1)$, which exists by Apisa-Wright \cite[Lemma 11.2]{ApisaWrightHighRank}. If surfaces in $\cM_v$ are connected, then, setting $\cM' = \cM_v$, we are done by Lemma~\ref{lemma:genusbytwo}.  The claim about marked points will be shown in the final paragraph of this proof.

Suppose therefore that $\cM_v$ consists of disconnected surfaces and therefore, by Lemma \ref{lemma:ConnCompsDeg}, $\cM_v$ is contained in a product of hyperelliptic loci and components of strata of rank at least two so that the projection of $\cM_v$ to each factor surjects  onto an open dense set.  In particular, using Lemma \ref{L:StillGeneric}, since the cylinders in $\Col_v(\bfC_1)$ and $\Col_v(\bfC_2)$ remain generic cylinders, these cylinders are simple and half-simple by Apisa \cite[Corollary 6.6]{Apisa-Lyapunov} or Apisa-Wright \cite{ApisaWrightMPts}.  Moreover, any typical degeneration supported on either $\bfC_1 - (\bfC_1)_v$ or $\bfC_2$ has its support become, on each component of $\Col_v(X, \omega)$ a subset of either a union of two simple cylinders or a union of two half-simple cylinders glued along a $2$-IET.  This observation is important since collapsing such configurations does not disconnect the surface.

Suppose first that either $\bfC_2$ is involved with rel or that $\mathrm{rel}(\cM) = 0$. Let $w$ be a typical element of the twist space of $\bfC_2$. As explained in the previous paragraph, $\Col_w(X, \omega)$ is connected. Since $w$ is typical, we may conclude by setting $\cM' = \cM_{w}$. The claim about free marked points will be shown in the final paragraph.

In particular, if there is any equivalence class $\bfC_3$ in the complement of $\bfC_1$ that is involved with rel, then, setting $\bfD_1 := \bfC_3$ and $\bfD_2 := \bfC_1$, the previous paragraph produces a typical degeneration $v$ in $\mathrm{Twist}(\bfD_2) = \mathrm{Twist}(\bfC_1)$ so that $\cM_v$ consists of connected surfaces. 

So it remains to consider the case where $\mathrm{rel}(\cM) > 0$ and $\mathrm{Twist}(\bfC_1)$ contains all the rel of $\cM$. In particular, for typical elements of the twist space, $\bfC_v$ is a proper subset of $\bfC_1$.  Setting $\bfC_1' := \bfC_1 - (\bfC_1)_v$, there is an  element $v'$ of $\mathrm{Twist}(\bfC_1)$ supported on $\bfC_1'$. As in \cite[Proof of Lemma 11.2]{ApisaWrightHighRank}\footnote{The proof notes that there is a collection of hyperplanes in the twist space containing elements $u$ so that, along $(X, \omega) + tu$, two cylinders whose heights are not generically constant multiples of each other have a constant ratio of heights. Any deformation off of this collection of hyperplanes can be used to form a typical degeneration.}, we may perturb $v'$ to a typical degeneration $w$ so that $\bfC_{w}$ is a subset of $\bfC_1'$. As in the previous paragraph, $\Col_w(X, \omega)$ is connected and, since $w$ is typical, we may conclude by setting $\cM' = \cM_{w}$.

We will now show that the surfaces in $\cM'$ contain no free marked points. If $\mathrm{rel}(\cM) \geq 1$, this is immediate from Apisa-Wright \cite[Lemma 11.6]{ApisaWrightHighRank} and the assumption that $\cM$ has no marked points. If $\cM$ is rel zero, then it follows by \cite[Lemma 10.5]{ApisaWrightHighRank}; note that the hypothesis that $\text{Col}_v(\mathbf{C}_1)$ contains at least two saddle connections holds by Lemma \ref{lemma:NonArithConnProdTwoSC}. 
\end{proof}



\begin{proof}[Proof of Theorem \ref{T:main} given Theorems \ref{T:main-tool} and \ref{T:Main-tool2}:] First we will show that if $\cM$ is an orbit closure so that $\For(\cM)$ is the $(1,1,1,7)$-locus, then the surfaces in $\cM$ are just collections of free marked points on surfaces in the $(1,1,1,7)$-locus. By Theorem \ref{T:main-tool}, the $(1,1,1,7)$-locus has no periodic points. By \cite{ApisaWrightMPts}, it therefore suffices to show that the $(1,1,1,7)$-locus, which we call $\cN$, does not belong to a (full) locus of covers $\cL$. If it did, then since full loci of covers have rational fields of definition, the tangent space of $\cL$ at a point $(X, \omega) \in \cN$ would contain both $p(T_{(X, \omega)} \cN)$ and its Galois conjugate, i.e., $\cL$ would be full rank. Since $\cN$ is contained in $\cH(6)$, this implies that $\cL$ is the hyperelliptic component of $\cH(6)$, which is not the component containing the $(1,1,1,7)$-locus, a contradiction. 



We will now show that if $\cM$ is algebraically primitive and has rank two, rel $r$, then $\For(\cM)$ is the $(1,1,1,7)$-locus. We will induct on $r$, with Theorems \ref{T:main-tool} and \ref{T:Main-tool2} forming the base cases. In particular, let $r > 1$. Suppose in order to derive a contradiction that $\cM$ is a counterexample to our claim. If $\For(\cM)$ has smaller rel than $\cM$, then $\For(\cM)$ is the $(1,1,1,7)$-locus by the induction hypothesis, so $\cM$ would not be a counterexample to the claim. Therefore, $\For(\cM)$ and $\cM$ have the same rel and so we will replace $\cM$ with $\For(\cM)$ in order to assume that the surfaces in $\cM$ have no marked points. By Proposition \ref{P:induction}, we may find an algebraically primitive invariant subvariety $\cM'$ that is rank two, rel $r-1$, and that consists of connected genus four surfaces with no free marked points. By the induction hypothesis, $\For(\cM')$ is the $(1,1,1,7)$-locus. Since $\mathrm{rel}(\cM') > 0$, the first paragraph implies that surfaces in $\cM'$ have free marked points, which is a contradiction. 




We will now show that there are no algebraically primitive invariant subvarieties with rank at least three and quadratic field of definition. We will induct on the dimension of $\cM$. Suppose first that $\cM$ has rank three rel zero. By forgetting marked points, we may suppose without loss of generality that $\cM$ has no marked points. By Proposition \ref{P:induction}, we may find an algebraically primitive invariant subvariety that is rank two, rel one, and that consists of connected genus four surfaces with no free marked points, contradicting Theorem \ref{T:Main-tool2}. This establishes the base case. The inductive step is similar, we forget marked points and apply Proposition \ref{P:induction} to produce a lower-dimensional algebraically primitive invariant subvariety with rank at least three and quadratic field of definition, which cannot exist by the induction hypothesis.
\end{proof}

\section{The unique nonarithmetic cylinder rigid invariant subvariety in genus two}\label{S:CylinderRigidInGenusTwo}

Recall that if $\cM$ is an invariant subvariety whose points are connected translation surfaces, then $\cM$ is called \emph{cylinder rigid} if every equivalence class can be partitioned into subequivalence classes, and if, additionally, subequivalent cylinders remain subequivalent as long as they persist under deformations that remain in $\cM$.  By definition, a \emph{subequivalence class} of cylinders is a collection of equivalent cylinders that have a constant ratio of heights (belonging to a fixed finite set only depending on $\cM$) under all deformations that remain in $\cM$ and so that the standard shear in this collection of cylinders belongs to the tangent space. Examples of cylinder rigid invariant subvarieties include rel zero invariant subvarieties and invariant subvarieties in their boundaries. The definition of cylinder rigid invariant subvarieties and a description of their properties appears in \cite{Apisa-Lyapunov}.

Before stating the main result of the section, we say that a \emph{cylinder configuration} of a cylinder rigid rank one rel one invariant subvariety is a horizontally periodic surface in the subvariety (defined up to shearing and dilating, i.e. applying cylinder deformations that are multiples of the standard shear, subequivalence classes). In the sequel $\phi$ will be the golden ratio. The main result of this section is the following.

\begin{figure}[h]
\centering
\includegraphics[width = .25\linewidth]{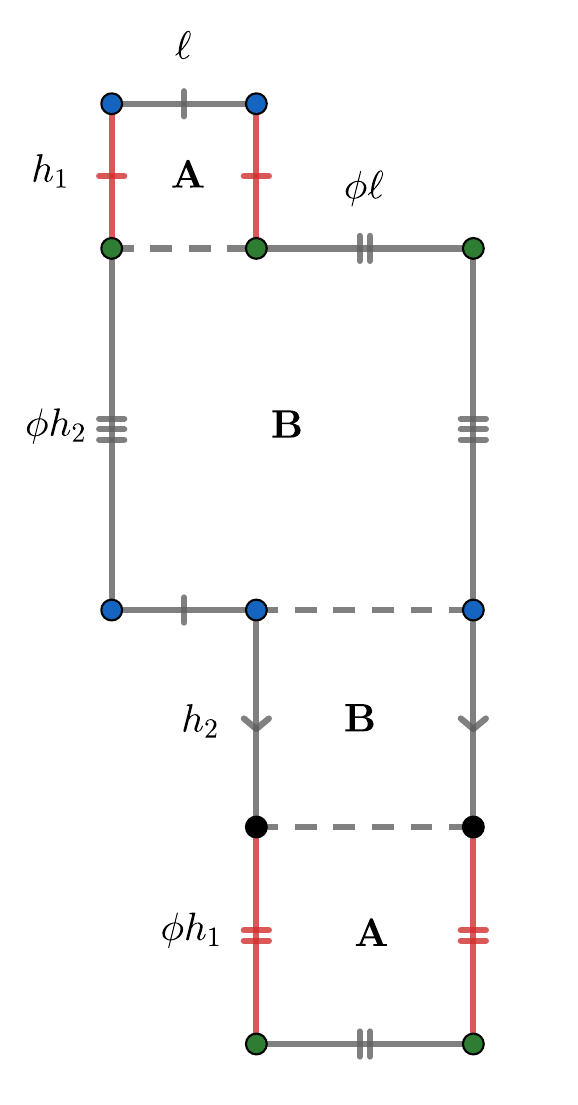}
\caption{A cylinder configuration in the golden eigenform locus with the golden point marked. The vertical and horizontal saddle connections are labelled by their lengths, which depend on free parameters $h_1, h_2$, and $\ell$.}
\label{F:GoldenCylConfig}
\end{figure}

\begin{proposition}\label{P:GenusTwoCylinderRigid}
If $\cM$ is a nonarithmetic cylinder rigid rank one rel one invariant subvariety of genus two surfaces, then $\cM$ is the golden eigenform locus with exactly one golden point marked and no other marked points. Moreover, the only cylinder configuration in this locus with two horizontal subequivalence classes is, up to rotating by $\pi$, the one in Figure \ref{F:GoldenCylConfig}.
%
\end{proposition}

Recall that, in the notation of McMullen \cite{McMullen-BilliardsTeichCurveSurvey}, the \emph{golden eigenform locus} is $\Omega E_5 \cap \cH(1,1)$, i.e., the locus of eigenforms for the action of real multiplication by the maximal order of $\bQ(\sqrt{5})$ in $\cH(1,1)$. The \emph{golden point} is a periodic point (and its image under the hyperelliptic involution) found on surfaces in the golden eigenform locus by Kumar-Mukamel \cite{KumarMukamel}. 

Recall that if $\cM$ is an invariant subvariety consisting of translation surfaces with marked points, then $\For(\cM)$ is the invariant subvariety formed by forgetting the marked points on the translation surfaces in $\cM$. 

\begin{lemma}\label{L:AlmostNoCylinderRigid1}
No rank one rel one cylinder rigid invariant subvariety $\cM$ has the property that $\For(\cM)$ is a nonarithmetic eigenform locus in $\cH(1,1)$ unless $\cM$ is the golden eigenform locus and a golden point is marked. 
\end{lemma}
\begin{proof}
The surfaces in $\cM$ have only periodic points marked. Fix a surface $(X, \omega)$ so that $\For(X, \omega)$ has three horizontal cylinders $\{C_i\}_{i=1}^3$ with core curves $\{\gamma_i\}_{i=1}^3$. Let $\gamma_i^*$ be the element of $H^1(X, Z; \bC)$, where $Z$ is the set of zeros of $\omega$, given by taking intersections with $\gamma_i$.

We will first show that the surfaces in $\cM$ must have some marked points. If not, then since $\gamma_1^* - \gamma_2^* + \gamma_3^* \in T_{(X, \omega)} \cM$, the cylinder rigidity of $\cM$ implies that $\gamma_1^* + \gamma_3^*$ and $\gamma_2^*$ also belong to the tangent space. For cylinder rigid invariant subvarieties, each element of the twist space on a horizontally periodic surface can be written as a linear combination of standard shears in subequivalence classes, see Apisa \cite{Apisa-Lyapunov}. It follows that $C_1$ and $C_3$ are subequivalent and have identical heights. Since they are subequivalent, they also have a rational ratio of moduli and so their ratio of circumferences is rational. Since the length of $\gamma_2$ is the sum of lengths of $\gamma_1$ and $\gamma_3$, we have that the ratio of any two circumferences of cylinders in $\{C_1, C_2, C_3\}$ is rational. By Wright's cylinder deformation theorem \cite{WrightCylDef}, $\cM$ has rational field of definition, which is a contradiction. The argument in the case where only Weierstrass points are marked is identical. Therefore, some periodic point that is not a Weierstrass point must be marked and this forces $\For(\cM)$ to be the golden eigenform locus and for the golden point to be marked on $\cM$ by Apisa \cite{ApisaPerPtsGen2}. 
\end{proof}

\begin{lemma}\label{L:NoH2}
No rank one rel one cylinder rigid invariant subvariety $\cM$ has the property that $\For(\cM)$ is the decagon locus or a nonarithmetic eigenform locus in $\cH(2).$ \end{lemma}
\begin{proof}
Fix a surface $(X, \omega)$ so that $\For(X, \omega)$ has two horizontal cylinders $\{C_1, C_2\}$. By Apisa \cite{ApisaPerPtsGen2}, which essentially relies on work of M\"oller \cite{Moller-PeriodicPoints}, since $\cM$ has rel one, the marked points consist of Weierstrass points and either one free point or two points exchanged by the hyperelliptic involution. In both cases, there is an index $i$ so that the only marked points in $C_i$ are Weierstrass points. The heights of the cylinders in $C_i$ are unchanged by the rel deformation. But this  contradicts the assumption that $\cM$ is cylinder rigid and rank one rel one.
\end{proof}

\begin{lemma}\label{L:AlmostNoCylinderRigid2}
If $\cM$ is a nonarithmetic cylinder rigid rank one rel one invariant subvariety of genus two surfaces, then $\cM$ is the golden eigenform locus with exactly one golden point marked and no other marked points. 
\end{lemma}
\begin{proof}
By Lemmas \ref{L:AlmostNoCylinderRigid1} and \ref{L:NoH2}, $\cM$ is the golden eigenform locus and at least one golden point is marked.  Let $p$ be a marked golden point.  Let $(X, \omega)$ be the surface in Figure~\ref{F:GoldenCylConfig}, and let $P$ be the points that are marked in addition to the golden point $p$. $P$ consists of a subset of Weierstrass points and possibly also the image $p'$ of $p$ under the hyperelliptic involution.  Let $C$ be the bottom-most horizontal cylinder with respect to the presentation in Figure~\ref{F:GoldenCylConfig} on $\For(X, \omega)$, which contains $p$.

Suppose in order to derive a contradiction that $P$ contains either $p'$ or a Weierstrass point $w$ in $C$.  Applying a cylinder deformation (e.g. in order to ensure that $\phi h_1 = h_2$ in the notation of Figure \ref{F:GoldenCylConfig}), we may move $p$ and $p'$ (resp. $p$ and $w$) towards each other in $C$ while remaining in $\cM$. This causes the points to collide without sending the height of any cylinder in $\cF(X, \omega)$ to $0$ since all cylinder heights are at least $\min(h_1,h_2)$. However, this deformation does cause a horizontal cylinder on $(X, \omega)$ to have its height go to zero. The cylinders that remain at the end of this deformation are all subequivalent. But this yields the contradiction that all cylinders on $(X, \omega)$ that are contained in the top-most or middle cylinders of $\cF(X, \omega)$ are subequivalent (a linear combination of standard shears in subequivalence classes cannot be rel if this is the case). We have shown that $P$ does not contain $p'$ or the two Weierstrass points in $C$.

Suppose in order to deduce a contradiction that $P$ contains a Weierstrass point, which will necessarily be one that is not in $C$. As before, move $p$ to a Weierstrass point $w$ in the interior of $C$ (this time $w$ is not marked), and let $w'$ be the Weierstrass point in $C$ that is different from $w$. This passes to a surface in the  orbit closure generated by the decagon.  By applying an element of $\genplin$ if necessary, we can pass to the regular decagon so that $p$ is the center of the decagon.  This is illustrated in \cite[Figure 3.1]{ApisaPerPtsGen2}.  Using the rotational symmetry of the decagon, any Weierstrass point not at the center of the decagon can be moved to any other. Therefore, we may move the marked Weierstrass point to $w'$. Moving $p$ off of $w$ and back to its original position in $(X, \omega)$, we have the contradiction that $p$ and a Weierstrass point in $C$, i.e., $w'$, are both marked\footnote{An equivalent argument that avoids the rotational symmetry of the decagon is that $p=w$ and any other Weierstrass point $w''$ can be joined by a straight line in the decagon. In the direction of this line, we have a three-cylinder decomposition of the decagon. Moving $p$ off of $w$ then puts $p$ and $w''$ into precisely the configuration we ruled out. A third, even more succinct argument is that we have shown that two Weierstrass points cannot be marked and the Veech group of the decagon acts transitively on five of the Weierstrass points.}. Therefore, $p$ is the only marked point, as desired.
%
%
\end{proof}

In the sequel, we will say that a subequivalence class $\bfC$ is in the \emph{golden configuration} if it contains two cylinders of equal moduli, one of which is simple, which share a boundary saddle connection and so that their ratio of circumferences is $\phi$ (see the cylinders labeled $\bfB$ in Figure \ref{F:GoldenCylConfig}). Note that the golden configuration describes how a subset of parallel cylinders are arranged. Since it does not describe how \emph{all} parallel cylinders in a given direction are arranged, it is not a cylinder configuration.

\begin{proof}[Proof of Proposition \ref{P:GenusTwoCylinderRigid}:] By Lemma \ref{L:AlmostNoCylinderRigid2}, it suffices to classify the cylinder configurations with two horizontal cylinder subequivalence classes, $\bfA$ and $\bfB$, in $\cM$, the golden eigenform locus with the golden point marked.  

Since $\cH(1,1,0)$ has at most four cylinders in any periodic direction, $\bfA$ and $\bfB$ both contain exactly two cylinders (otherwise, $\cM$ would be arithmetic, as can be seen by using a standard shear in one subequivalence class to collapse the surface onto the one-cylinder subequivalence class and applying the cylinder deformation theorem \cite{WrightCylDef}).

There are two four-cylinder configurations in $\cH(1,1,0)$ - the one shown in Figure \ref{F:GoldenCylConfig} and one, called the \emph{second configuration}, in which the marked point has been moved to the middle cylinder. See Figure \ref{F:SecondConfig}. In both cases, the two subequivalent cylinders cannot border one another along a saddle connection joining the marked point to itself by the argument in the previous paragraph.

By this observation, for the second configuration, both $\bfA$ and $\bfB$ contain a simple cylinder.  We show that the second configuration is impossible.  Let $a$ (resp. $b$) be their positively oriented core curves in (absolute) homology. The standard shear in $\bfA$ and $\bfB$ is a positive real linear combination of $a^*$ (resp. $b^*$) and $(a+b)^*$ where the $*$ denotes the Poincare dual. A nontrivial linear combination of these standard shears is rel, i.e. zero in homology. But this is impossible since if $\alpha$ and $\beta$ are linearly independent vectors in a real vector space and $v$ (resp. $w$) is a \textit{positive} linear combination of $\alpha$ (resp. $\beta$) and $\alpha + \beta$, then $v$ and $w$ are not collinear. So, up to relabeling, $\bfA$ and $\bfB$ are in the cylinder configuration shown in Figure \ref{F:GoldenCylConfig}.

\begin{figure}
\centering
    \begin{tikzpicture}[x=1cm,y=1cm,line width=.7pt,
  zeroone/.style={circle,draw=black,fill=blue!65,inner sep=1.7pt},
  zerotwo/.style={circle,draw=black,fill=green!55!black,inner sep=1.7pt},
  marked/.style={circle,fill=black,inner sep=1.7pt}]
  \fill[blue!6] (0,3) rectangle (2,4);
  \fill[orange!10] (0,2) rectangle (5,3);
  \fill[blue!6] (0,1) rectangle (5,2);
  \fill[orange!10] (2,0) rectangle (5,1);

  \draw (0,1)--(0,4)--(2,4)--(2,3)--(5,3)
        --(5,0)--(2,0)--(2,1)--cycle;
  \draw[dashed] (0,3)--(2,3);
  \draw[dashed] (0,2)--(5,2);
  \draw[dashed] (2,1)--(5,1);

  \foreach \y in {1,4}
    \draw (1,\y-.09)--(1,\y+.09);
  \foreach \y in {0,3}
    \foreach \x in {3.44,3.56}
      \draw (\x,\y-.09)--(\x,\y+.09);

  \foreach \xl/\xr/\y in {0/2/3.5,0/5/2.5,0/5/1.5,2/5/.5}{
    \foreach \x in {\xl,\xr}
      \draw (\x-.065,\y-.07)--(\x,\y+.04)
            --(\x+.065,\y-.07);
  }

  \node at (1,3.5) {$A$};
  \node at (2.5,2.5) {$B$};
  \node at (2.5,1.5) {$A$};
  \node at (3.5,.5) {$B$};

  \foreach \p in {(0,4),(2,4),(0,1),(2,1),(5,1)}
    \node[zeroone] at \p {};
  \foreach \p in {(0,3),(2,3),(5,3),(2,0),(5,0)}
    \node[zerotwo] at \p {};

  \node[marked,label=left:{$p$}] at (0,2) {};
  \node[marked,label=right:{$p$}] at (5,2) {};
\end{tikzpicture}
\caption{The second configuration}
\label{F:SecondConfig}
\end{figure}

To see that $\bfB$ is in the golden configuration, collapse $\bfA$. This passes, up to marked points, to the golden eigenform locus in $\cH(2)$. This locus is a single-cusp Teichm\"uller curve and so $\Col_\bfA(\bfB)$ and hence $\bfB$ must be in the golden configuration.

From this it is automatic that the ratio of circumferences of cylinders in $\bfA$ is the golden ratio. In particular, any two subequivalent cylinders have the property that one is golden ratio times longer than the other. To conclude it suffices to show that if one cylinder in $\bfA$ has a vertical cross curve, then the other does as well. We will refer to the two cylinders in $\bfA$ as the ``top'' and ``bottom'' cylinder, referring to their positions in Figure \ref{F:GoldenCylConfig}. 

First, if the bottom $\bfA$-cylinder has a vertical cross curve, but not the top one, then collapsing $\bfA$ causes the golden point to merge with a zero without leaving $\cH(1,1)$. The resulting locus is necessarily a Teichm\"uller curve in the golden eigenform locus and hence the decagon locus. However, the decagon is a two-cusp Teichm\"uller curve and exactly one of those cusps corresponds to a $2$-cylinder direction.  But the shorter cylinder in that direction is not simple, whereas the shorter cylinder in $\bfB$ after collapsing $\bfA$ is, which is a contradiction.

Finally, if the top $\bfA$-cylinder has a vertical cross curve, but not the bottom one, then collapsing $\bfA$ passes to a surface in the golden eigenform locus in $\cH(2,0)$ with a periodic point marked on the boundary of a cylinder in $\Col_{\bfA}(\bfB)$ that becomes simple after forgetting marked points. The marked point is necessarily a Weierstrass point by M\"oller \cite{Moller-PeriodicPoints}. But for a two-cylinder direction for a surface in $\cH(2)$, Weierstrass points do not appear on the boundary of simple cylinders, a contradiction.
\end{proof}

\section{Proof of Theorem \ref{T:main-tool}}\label{S:Rk2Rel0}

Suppose that $\cM$ is a rank two rel zero algebraically primitive invariant subvariety consisting of genus four surfaces. Note that we permit the surfaces in $\cM$ to have marked points. By Apisa-Wright \cite[Lemma 3.31]{ApisaWrightDiamonds}, we may find a surface $(X, \omega) \in \cM$ with two generic equivalence classes of cylinders $\bfC_1$ and $\bfC_2$ that do not intersect and which share no boundary saddle connections. Without loss of generality, $\bfC_1$ (resp. $\bfC_2$) consists of vertical (resp. horizontal) cylinders. We will use the following facts:
\begin{itemize}
    \item $\cM_{\bfC_i}$ is cylinder rigid for $i \in \{1,2\}$ by Apisa \cite[Lemma 4.2 and Proposition 4.12]{Apisa-Lyapunov},
    \item $\cM_{\bfC_i}$ is algebraically primitive for $i \in \{1,2\}$ by Lemma \ref{lemma:genusbytwo}.
\end{itemize}

\begin{lemma}\label{L:OneConnected}
At least one of $\cM_{\bfC_1}$ and $\cM_{\bfC_2}$ is connected. If $\cM_{\bfC_i}$ is connected, then it is the golden eigenform locus with one golden point marked.
\end{lemma}
\begin{proof}
The second claim is immediate from Proposition \ref{P:GenusTwoCylinderRigid} so we will concentrate on the first claim. Suppose that $\cM_{\bfC_2}$ is disconnected. By Lemma \ref{lemma:genusbytwo}, $\cM_{\bfC_2}$ belongs to a product of two genus one strata. Since $\cM_{\bfC_2}$ is rank one rel one and cylinder rigid, $\Col_{\bfC_2}(X, \omega)$ has two subequivalence classes of vertical cylinders. Let $\bfC_3$ be the equivalence class of vertical cylinders on $(X, \omega)$ so that $\Col_{\bfC_2}(\bfC_1)$ and $\Col_{\bfC_2}(\bfC_3)$ are the two vertical subequivalence classes of cylinders on $\Col_{\bfC_2}(X, \omega)$. Since $(X, \omega)$ is connected, there is a cylinder $C \in \bfC_2$ and two cylinders $A, B \in \bfC_3$ so that $\Col_{\bfC_2}(A)$ and $\Col_{\bfC_2}(B)$ belong to different components and $C$ intersects both $A$ and $B$. 
To see this, color all the vertical cylinders (and their boundary saddle connections) in $\bfC_1$ and $\bfC_3$ that end up on one component of $\Col_{\bfC_2}(X, \omega)$ red. Color the ones that end up on the other component blue. By connectedness, some vertical saddle connection $s$ must be colored red and blue. Since $s$ would connect the red and blue components, it must disappear when $\bfC_2$ collapses and hence be contained therein. Since $\bfC_2$ is contained in $\bfC_3$, $s$ must appear on the boundary of one red and one blue cylinder in $\bfC_3$ as desired.

\color{black}


The surface $\Col_{\bfC_1, \bfC_2}(X, \omega)$ has two components and $\Col_{\bfC_1, \bfC_2}(A)$ and $\Col_{\bfC_1, \bfC_2}(B)$ are found on different components. If $\Col_{\bfC_1}(X, \omega)$ had two components, then $\Col_{\bfC_1}(A)$ and $\Col_{\bfC_1}(B)$ would be found on different components; however this is not the case since they both intersect $\Col_{\bfC_1}(C)$.  
%
%
%
\end{proof}

In light of Lemma \ref{L:OneConnected}, we may suppose without loss of generality that $\cM_{\bfC_1}$ is connected.


\begin{lemma}\label{L:DisconnectedRank2Rel0Case}
If $\cM_{\bfC_2}$ is disconnected, then $\cM$ is the $(1,1,1,7)$-locus. Moreover, $\bfC_1$ is a pair of simple cylinders and $\bfC_2$ is a pair of cylinders in the golden configuration. 
\end{lemma}
\begin{proof}
By Proposition \ref{P:GenusTwoCylinderRigid}, $\Col_{\bfC_1}(\bfC_2)$ is either a pair of simple cylinders or a pair of cylinders in the golden configuration. If $\cM_{\bfC_2}$ is disconnected, then the second case is the only possible one. Since $\Col_{\bfC_1}(\bfC_1)$ consists of vertical saddle connections in the complement of the horizontal cylinders $\Col_{\bfC_1}(\bfC_2)$, we see that $\Col_{\bfC_1}(\bfC_2)$ is as pictured in Figure \ref{F:example:sub1}.

\begin{figure}[H]
\centering
\begin{subfigure}{.35\textwidth}
  \centering
  \includegraphics[width=.8\linewidth]{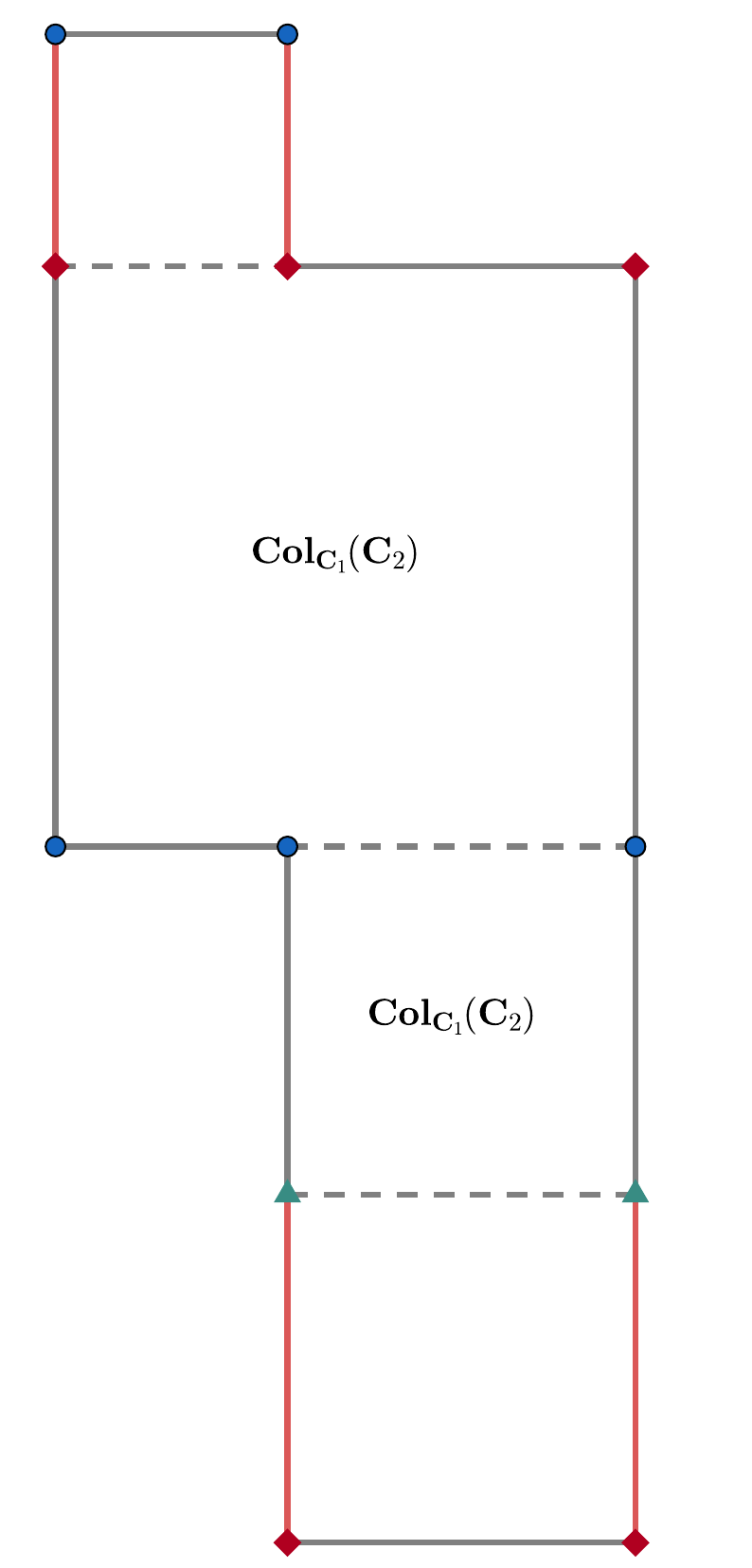}
  \caption{The surface $\Col_{\bfC_1}(X, \omega)$ when $\Col_{\bfC_2}(X, \omega)$ is disconnected.}
  \label{F:example:sub1}
\end{subfigure}
\hspace{.5cm}
\begin{subfigure}{.55\textwidth}
  \centering
  \includegraphics[width=.8\linewidth]{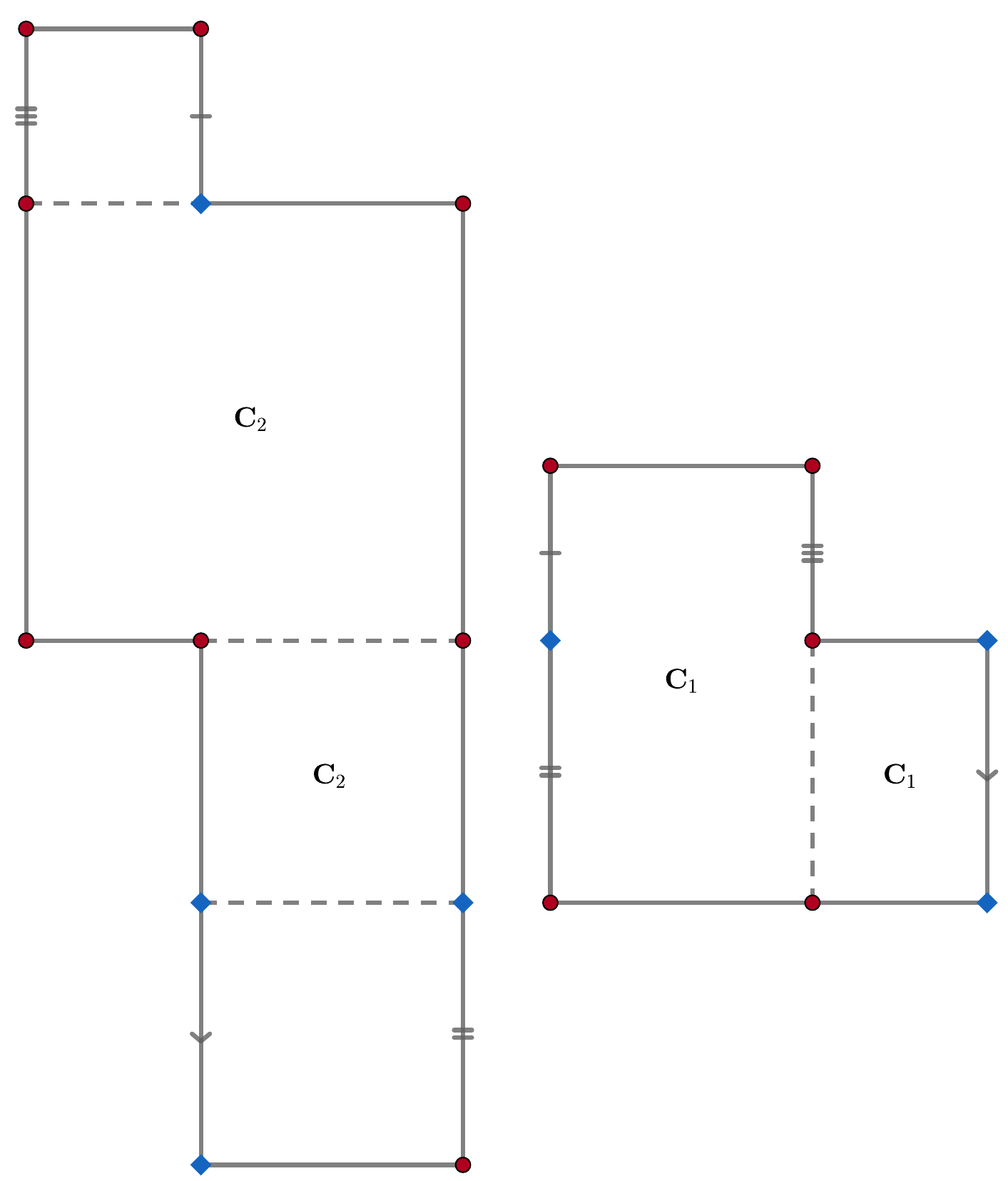}
  \caption{The surface $\Col_{\bfC_1}(X, \omega)$ when $\Col_{\bfC_2}(X, \omega)$ is connected and $\Col_{\bfC_1}(\bfC_2)$ is in the golden configuration. The surface is also potentially the mirror image of this one.}
  \label{F:example:sub2}
\end{subfigure}
\caption{These are the two degenerations that occur when there is a generic equivalence class in the golden configuration.}
\label{F:example}
\end{figure}

We see therefore that $\cM_{\bfC_2}$ consists of surfaces made up of two genus one components (note that each component may have marked points). On each component there are two subequivalence classes of vertical cylinders, one of which is $\Col_{\bfC_2}(\bfC_1)$. From Figure \ref{F:example:sub1}, we see that the second vertical subequivalence class on $\Col_{\bfC_2}(X, \omega)$, which we write as $\Col_{\bfC_2}(\bfC_3)$, where $\bfC_3$ is an equivalence class of vertical cylinders on $(X, \omega)$, consists of one simple cylinder on each component of $\Col_{\bfC_2}(X, \omega)$.  Every vertical subequivalence class must border a cylinder in a different vertical subequivalence class, in particular each cylinder in $\bfC_1$ must share a boundary saddle connection with a cylinder in $\bfC_3$ (see \cite[Lemma 4.4]{Apisa-Lyapunov}). Since the cylinders in $\Col_{\bfC_2}(\bfC_3)$ are simple,   there are at most two cylinders in $\Col_{\bfC_2}(\bfC_1)$ on each component of $\Col_{\bfC_2}(X, \omega)$ and these are necessarily simple. If there is a single cylinder in $\Col_{\bfC_2}(\bfC_1)$ on each component, then gluing these cylinders into the unique pair of vertical saddle connections in the complement of $\Col_{\bfC_1}(\bfC_2)$ on $\Col_{\bfC_1}(X, \omega)$ produces a surface in the $(1,1,1,7)$-locus. 

We will therefore suppose in order to deduce a contradiction that there are two cylinders in $\Col_{\bfC_2}(\bfC_1)$ on a single component of $\Col_{\bfC_2}(X, \omega)$, which means that there is a marked point between them. Shearing these cylinders in order to perform one Dehn twist in the cylinders in $\For(\bfC_1)$ will move this marked point off of a horizontal saddle connection in $\For(\bfC_1)$. However, after shearing, $\Col_{\bfC_1}(X, \omega)$ has an extra vertical saddle connection (since the marked point no longer coincides with a zero). This contradicts Proposition \ref{P:GenusTwoCylinderRigid}.     
\end{proof}

\begin{proof}[Proof of Theorem \ref{T:main-tool}:] Suppose not. By Lemma \ref{L:DisconnectedRank2Rel0Case} and since any cylinder equivalence class may be perturbed to form a generic diamond, we may suppose that for any generic cylinder equivalence class $\bfC$, $\cM_{\bfC}$ is connected.

\begin{sublem}\label{SL:AllSimple}
Any generic cylinder equivalence class $\bfC$ consists of two simple cylinders with one golden-ratio larger than the other.
\end{sublem}
\begin{proof}
Suppose that this is not the case and let $\bfC$ be a generic equivalence class of cylinders that contradicts the claim. Perturb the surface so that there is a generic equivalence class $\bfD$ of cylinders that are disjoint from those in $\bfC$. By assumption, $\cM_{\bfD}$ consists of connected surfaces, so $\cM_{\bfD}$ is the golden eigenform locus with the golden point marked. In particular, by Proposition \ref{P:GenusTwoCylinderRigid}, $\bfC$ is in the golden configuration (see Figure \ref{F:example:sub1}; when referring to figures throughout this paragraph, it will be useful to set $\bfC_1 := \bfD$ and $\bfC_2 := \bfC$).  This means that $\Col_{\bfC, \bfD}(X, \omega)$ is disconnected.

We will now show that $\bfD$ is also a pair of cylinders in the golden configuration.  Since $\cM_{\bfC}$ is also the golden eigenform locus with a golden point marked, $\Col_{\bfC}(\bfD)$, and hence $\bfD$ consists of either two simple cylinders or two cylinders in the golden configuration.  If $\bfD$ is a pair of simple cylinders, then the same is true of $\Col_{\bfC}(\bfD)$. But then $\Col_{\bfC, \bfD}(X, \omega)$ is connected since $\Col_{\bfC}(X, \omega)$ is connected by assumption and collapsing a pair of simple cylinders cannot disconnect a surface. But this is a contradiction. 

Since $\bfD$ consists of two cylinders in the golden configuration,  $(X, \omega)$ is the surface depicted\footnote{In particular, $\bfC_1$, which is in the golden configuration, has a boundary consisting of two pairs of equal length vertical saddle connections that must be glued to the saddle connections in the complement of $\Col_{\bfC_1}(\bfC_2)$ on $\Col_{\bfC_1}(X, \omega)$; these saddle connections in the complement have different lengths, which forces which boundary saddle connection of $\bfC_1$ they are glued to.} in Figure \ref{F:example:sub2}, which lies in $\cH(2, 4)$. We derive a contradiction by shearing $\bfC$ to produce an equivalence class $\bfC_3$ consisting of two simple cylinders (see Figure \ref{F:example2:sub2}). Collapsing it must result in the golden eigenform locus with the golden point marked. However, it is not possible to collapse two simple cylinders in $\cH(2, 4)$ and degenerate to $\cH(1,1,0)$ because surfaces in $\cH(2, 4)$ have a total of $16\pi$ in cone angle at the singularities, those in $\cH(1,1,0)$ have total cone angle $10\pi$, and collapsing two simple cylinders only results in a total loss of $4\pi$ in angle.
\end{proof}

\begin{figure}[H]
\begin{subfigure}{.55\textwidth}
\centering
  \includegraphics[width=.4\linewidth]{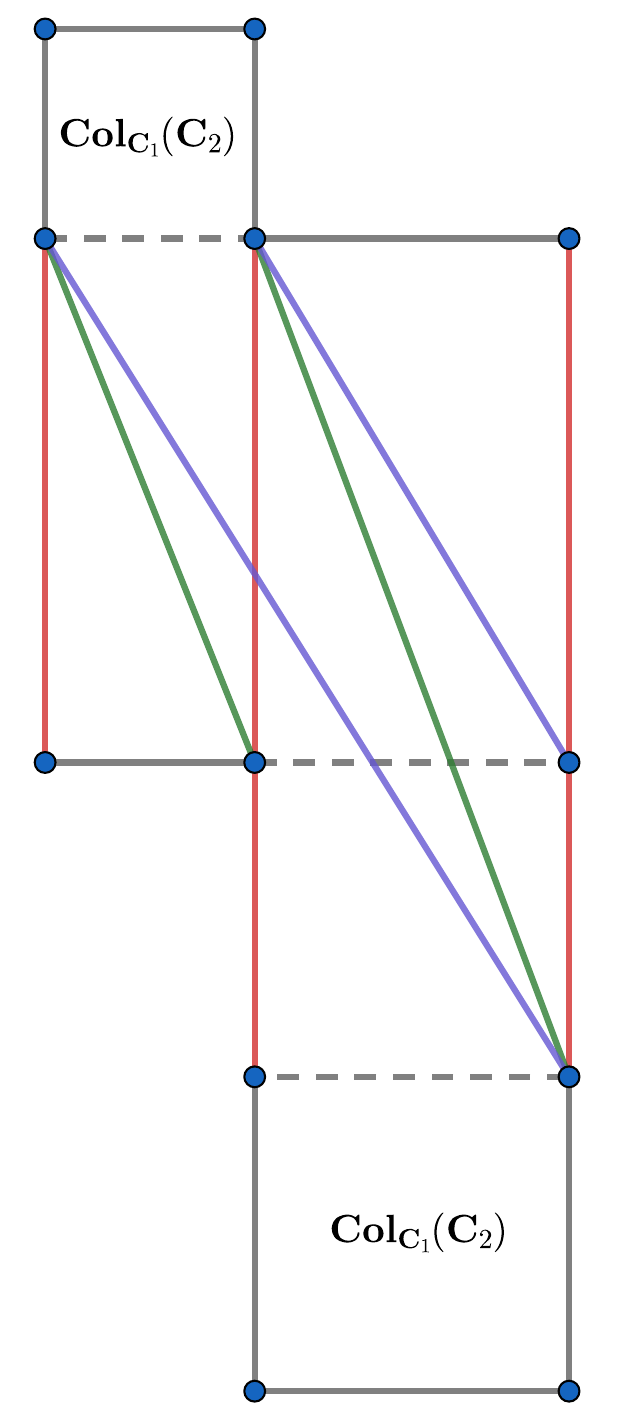}
\caption{$\Col_{\bfC_1}(\bfC_1)$ consists of two parallel saddle connections contained in the pair of horizontal cylinders in the golden configuration.}
  \label{F:example2:sub1}
\end{subfigure}
\hspace{.5cm}
\begin{subfigure}{.4\textwidth}
  \centering
  \includegraphics[width=.9\linewidth]{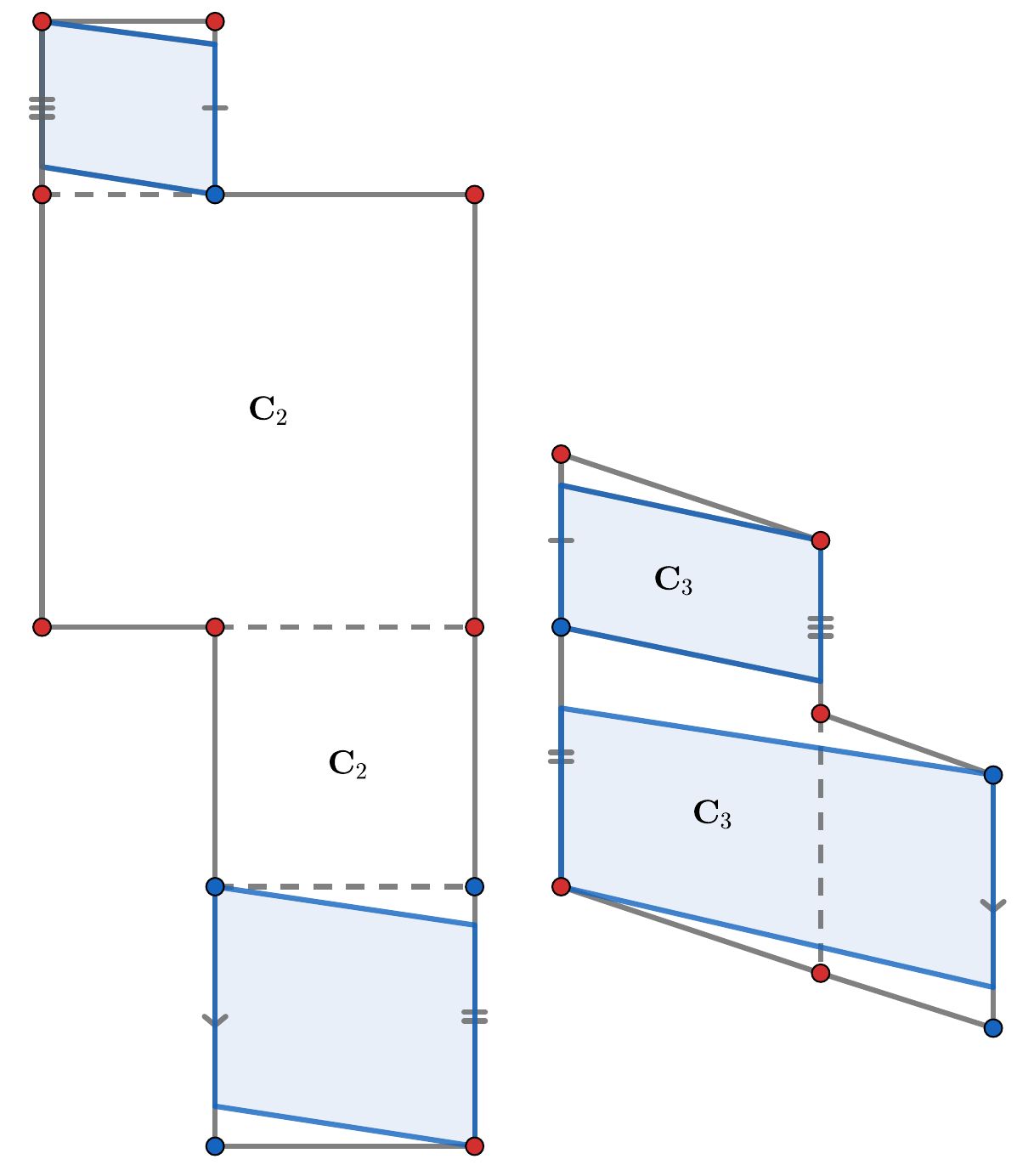}
  \caption{A contradiction.}
  \label{F:example2:sub2}
\end{subfigure}
\caption{The surface in the proof of Theorem \ref{T:main-tool}. Opposite sides of the polygon are identified unless otherwise specified.}
\label{F:example2}
\end{figure}

\begin{sublem}\label{SL:GoldenSC}
The surface $\Col_{\bfC_1}(X, \omega)$ is a horizontally periodic surface in the golden eigenform locus with one golden point marked and $\Col_{\bfC_1}(\bfC_1)$ is a pair of saddle connections in a pair of horizontal cylinders in the golden configuration. Up to shearing the golden configuration, $\Col_{\bfC_1}(\bfC_1)$ could be one of four pairs of saddle connections (see Figure \ref{F:example2:sub1}).
\end{sublem}
\begin{proof}
The first claim follows since $\bfC_1$ and $\bfC_2$ are disjoint pairs of simple cylinders. Let $\bfD$ be the horizontal cylinders in the golden configuration on $\Col_{\bfC_1}(X, \omega)$. Since $\Col_{\bfC_1}(\bfC_1)$ consists of two saddle connections, one golden ratio times longer than the other, it follows that the shorter saddle connection $s \in \Col_{\bfC_1}(\bfC_1)$ is a cross curve of a cylinder in $\bfD$. Let $t$ be the longer saddle connection. Note that the two cylinders in the golden configuration have identical moduli. If $s$ belongs to the simple cylinder in $\bfD$, then apply a shear to make it a vertical cross curve; so $t$ is a vertical cross curve of the non-simple cylinder in $\bfD$ of which there are two. Otherwise, $s$ is a cross curve for the non-simple cylinder in $\bfD$. By shearing to perform Dehn twists we may assume that its real part is less than the circumference of $\bfD$ and that its slope is negative.  There are two such cross curves and each has a corresponding unique choice of $t$.  
\end{proof}

At this point, the proof proceeds by rote. We consider the four pairs of saddle connections that could correspond to $\Col_{\bfC_1}(X, \omega)$, glue in simple cylinders to them, and assess whether the resulting surface can belong to a rank two rel zero orbit closure. We devise a criterion for ruling out two cases.

\begin{sublem}\label{SL:NoDiagonal}
Let $\bfC_3$ be an equivalence class of cylinders on $(X, \omega)$ so that $\bfC_2 \subseteq \bfC_3$. Then on $\Col_{\bfC_2}(X, \omega)$, $\Col_{\bfC_2}(\bfC_3)$ is not a pair of simple cylinders.

Moreover, the two saddle connections in $\Col_{\bfC_1}(\bfC_1)$ are each cross curves of a cylinder in the golden configuration\footnote{Gluing a pair of simple cylinders into the green pair of diagonal lines in Figure~\ref{F:example2:sub1} produces surfaces in the $(1,1,1,7)$-locus. The same is true for the purple lines. However, there is no convenient reference for this fact. So Sublemma \ref{SL:NoDiagonal} excludes these cases by using a standing assumption in the proof of Theorem \ref{T:main-tool}, namely that all equivalence classes are assumed to consist of a pair of simple cylinders.}. 
\end{sublem}

\begin{figure}[H]
\begin{subfigure}{.5\textwidth}
\centering
  \includegraphics[width=.8\linewidth]{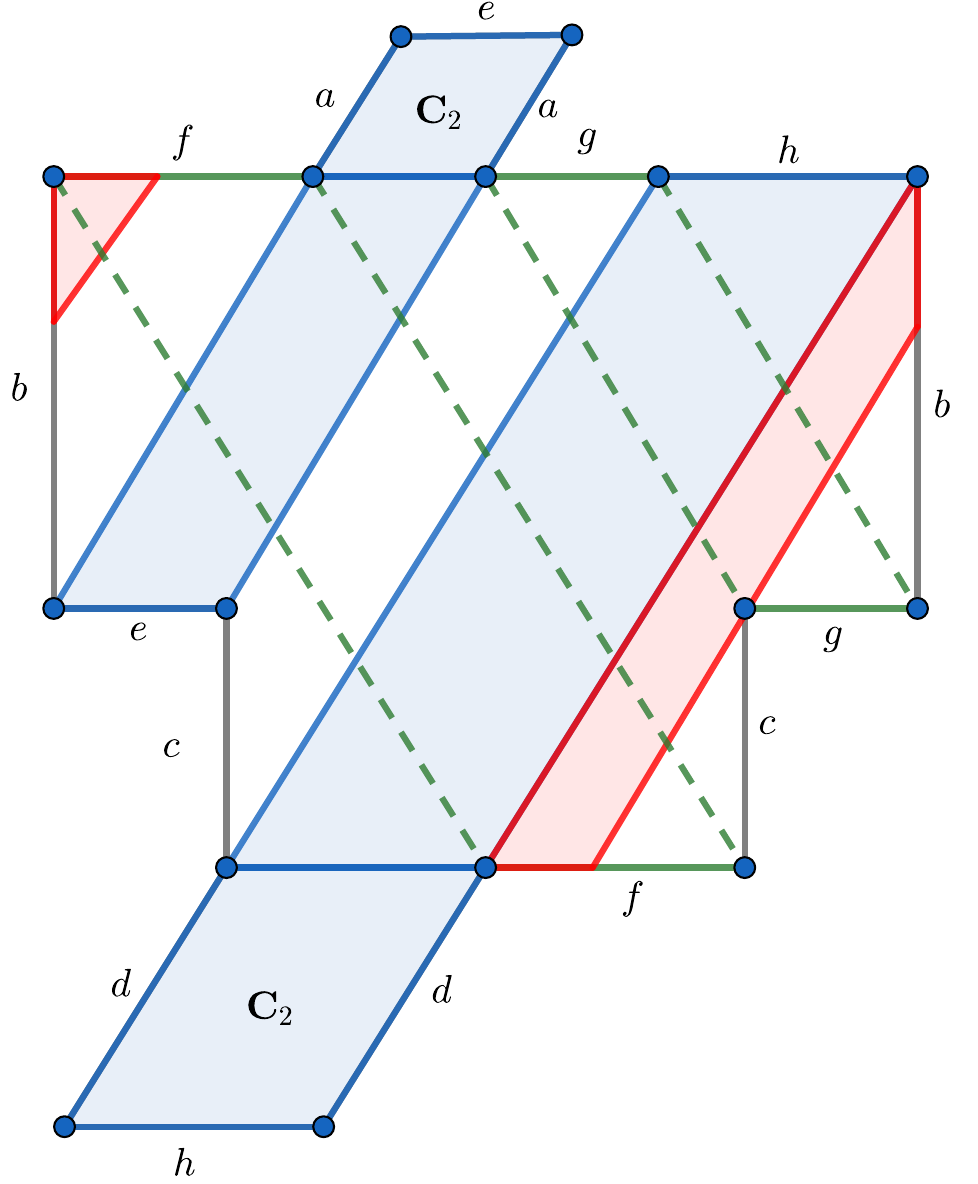}
\caption{$\cM$ is defined by the equations $d = \phi a$, $b = \phi c$, $f = \phi g$, and $h = \phi e$.}
  \label{F:example3:sub1}
\end{subfigure}
\hspace{.5cm}
\begin{subfigure}{.5\textwidth}
  \centering
  \includegraphics[width=.8\linewidth]{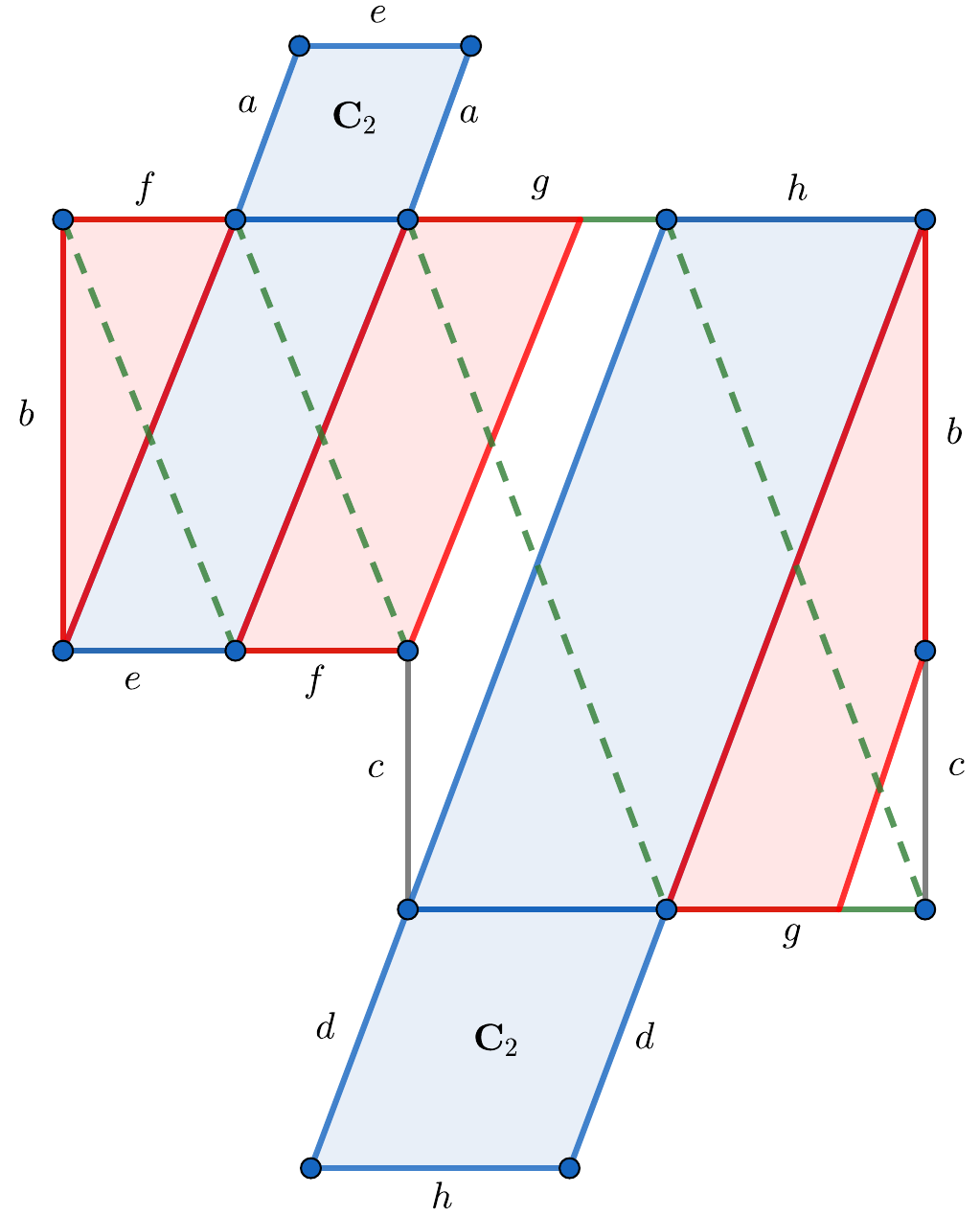}
\caption{$\cM$ is defined by the equations $d = \phi a$, $b = \phi c$, $g = \phi f$, and $h = \phi e$.}  \label{F:example3:sub2}
\end{subfigure}
\caption{The dashed (green) lines are the boundaries of the cylinders in $\bfC_1$. The cylinders in $\bfC_2$ are contained in $\bfC_3$, i.e. the two shaded (blue) cylinders. There are two more cylinders, i.e. those in $\bfC_4$, parallel to the shaded (blue) cylinders - one shaded (red) and one unshaded (white), which form a subequivalence class. }
\label{F:example3}
\end{figure}

\begin{proof}
Suppose not. There is an equivalence class $\bfC_4$ of cylinders on $(X, \omega)$ so that $\Col_{\bfC_2}(\bfC_4)$ is a subequivalence class of cylinders parallel to those in $\Col_{\bfC_2}(\bfC_3)$ on $\Col_{\bfC_2}(X, \omega)$. By Proposition \ref{P:GenusTwoCylinderRigid}, the cylinders in $\Col_{\bfC_2}(\bfC_4)$ and hence in $\bfC_4$ are in the golden configuration, contradicting Sublemma \ref{SL:AllSimple}. This proves the first claim. 

For the second claim, suppose in order to derive a contradiction that $\Col_{\bfC_1}(\bfC_1)$ consists of a pair of diagonal lines in the horizontal cylinders in the golden configuration in Figure~\ref{F:example2:sub1}. Gluing in a pair of simple cylinders to the indicated saddle connections produces the surfaces in Figure \ref{F:example3}. Take $\bfC_3$ to be the pair of vertical cylinders in Figure~\ref{F:example2:sub1} and note that they contain $\bfC_2$ and become a disjoint pair of simple cylinders on $\Col_{\bfC_2}(X, \omega)$ as shown in Figure \ref{F:example3}. This is a contradiction.
\end{proof}

Returning to the proof of Theorem~\ref{T:main-tool}, suppose finally that $\Col_{\bfC_1}(\bfC_1)$ consists of a pair of vertical saddle connections in the horizontal cylinders in the golden configuration in Figure~\ref{F:example2:sub1}. The surface $(X, \omega)$, up to a reflection about the vertical axis, is depicted in Figure \ref{F:example4:sub1}. \begin{sublem}\label{SL:RulingOutFigure5}
The surface in Figure \ref{F:example4:sub1} does not generate a rank two rel zero orbit closure.
\end{sublem} 
\begin{proof}
Notice that as long as they persist, the ratio of heights of the two cylinders in $\bfC_1$ will be the golden ratio. We will deform the surface to violate this condition. First, overcollapse the cylinders in $\bfC_2$, i.e. shear the cylinders so that they contain no vertical saddle connection and then send their heights first to zero and then slightly negative, as depicted in Figure~\ref{F:example4:sub2}.

\begin{figure}[H]
\begin{subfigure}{.4\textwidth}
\centering
  \includegraphics[width=.8\linewidth]{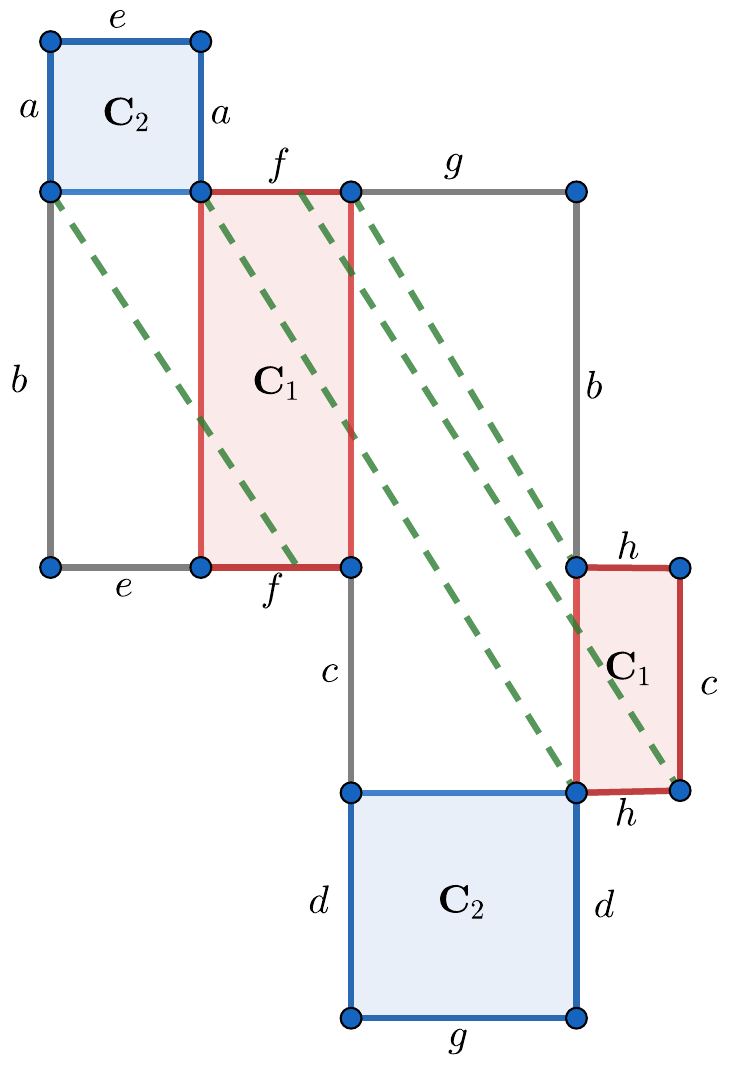}
\caption{$\cM$ is defined by the equations $d = \phi a$, $b = \phi c$, $f = \phi h$, and $g = \phi e$.}
  \label{F:example4:sub1}
\end{subfigure}
\hspace{.5cm}
\begin{subfigure}{.6\textwidth}
  \centering
  \includegraphics[width=.8\linewidth]{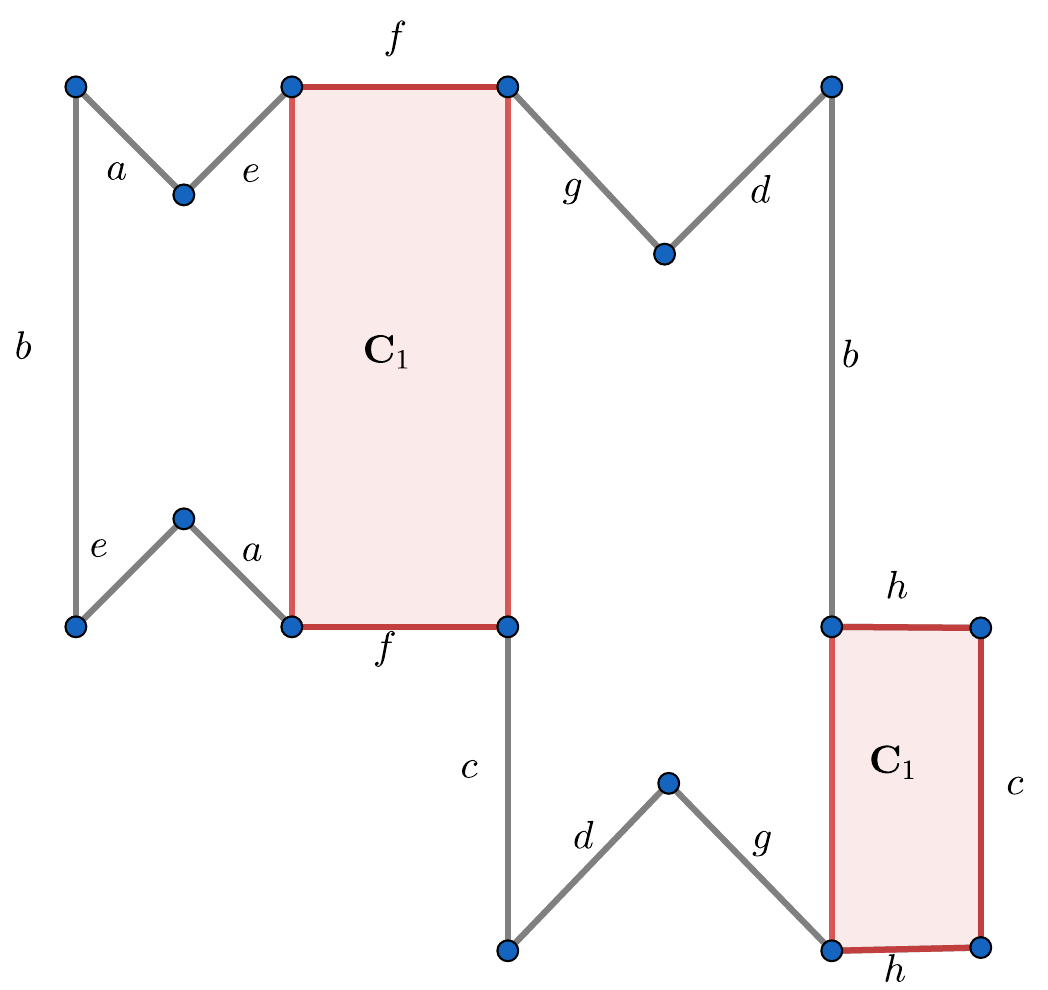}
\caption{The overcollapse of $\bfC_2$. The defining equations are $d = \phi e$, $b = \phi c$, $f = \phi h$, and $g = \phi a$.}  \label{F:example4:sub2}
\end{subfigure}
\caption{The surface constructed when $\Col_{\bfC_1}(\bfC_1)$ consists of a pair of vertical saddle connections and a degeneration.}
\label{F:example4}
\end{figure}

\begin{figure}[H]
\centering
\includegraphics[width=.5\linewidth]{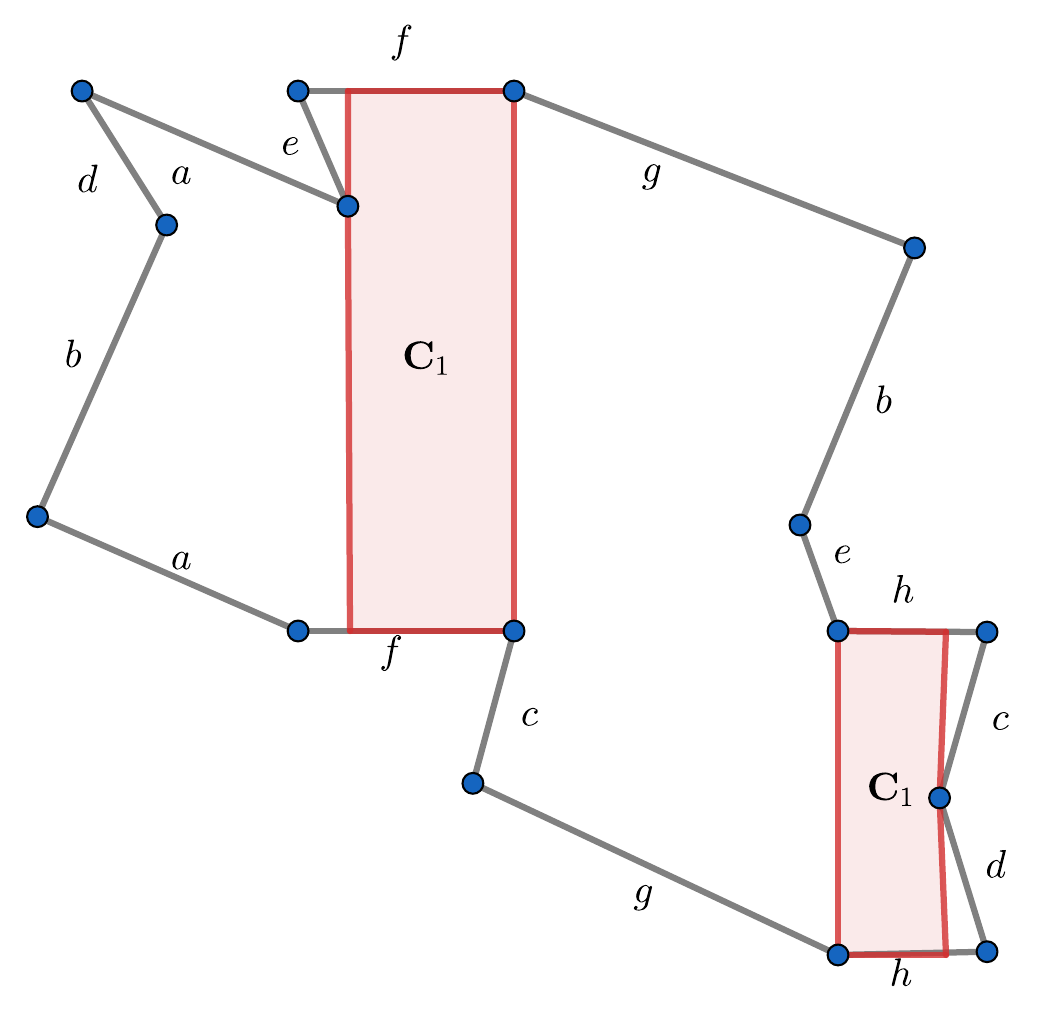}
\caption{Attacking.}
\label{F:example4b}
\end{figure}

Now we will change the periods of $a,d,e,g$ while keeping $a+e$ and $d+g$ constant (and while remaining in the orbit closure). Heuristically, this is the result of moving the tip of the triangle formed by the vectors $a$ and $e$ along a horizontal line while fixing the base, i.e. the vector $a+e$; and similarly for $d$ and $g$. Parameterize this deformation by a time $t$ so that if the period of $e$ is initially $e_0$ then at time $t$ it is $e_0 -t$. It follows that at time $t$, $d$ has period $\phi(e_0 - t)$. In particular, after $e$ passes through the vertical (see Figure \ref{F:example4b}) the height of the longer cylinder in $\bfC_1$ is $\phi h - t$ and the height of the shorter one is $h - \phi t$. But the ratios of these heights as a function of $t$ is not constant, which is a contradiction. 
\end{proof}
\color{black}

\end{proof}

\section{Proof of Theorem \ref{T:Main-tool2}}\label{S:Rk2Rel1}

To prove Theorem \ref{T:Main-tool2}, it suffices to show that no algebraically primitive rank two rel one invariant subvariety consisting of genus four surfaces with no marked points exists.  

Indeed, if such a subvariety with marked points existed, then forgetting the marked points would produce a rank two rel zero algebraically primitive invariant subvariety and hence the $(1,1,1,7)$ locus by Theorem \ref{T:main}. There are no periodic points in this locus by Theorem \ref{T:main} and, since the $(1,1,1,7)$ locus is not contained in a locus of covers, it follows by Apisa-Wright \cite{ApisaWrightMPts} that any marked point on the surfaces in this subvariety is a free marked point on surfaces in the $(1,1,1,7)$ locus. 

Suppose in order to derive a contradiction that $\cM$ is such an algebraically primitive rank two rel one invariant subvariety consisting of genus four surfaces with no marked points.

\begin{lemma}
\label{L:GenEqImpNoRel}
If $\bfC$ is a generic equivalence class on a surface $(X, \omega)$ in $\cM$, then the twist space of $\bfC$ does not contain rel.
\end{lemma}
\begin{proof}
Suppose not. Let $\bfC = \bfC_1$, and let $\bfC_2$ be a disjoint generic equivalence class (which exists, perhaps after perturbing, by Apisa-Wright \cite[Lemma 3.31]{ApisaWrightDiamonds}).

\begin{sublemma}
There is a typical degeneration $v$ in $\bfC_1$ so that $\Col_v(X, \omega)$ is connected and so $\cM_v$ is the $(1,1,1,7)$-locus.
\end{sublemma}
\begin{proof}
Given the first claim, the second claim follows immediately from Theorem~\ref{T:main-tool}, so we will concentrate on the first claim. Let $w$ be a typical degeneration in $\bfC_1$, which exists by \cite[Lemma 11.2]{ApisaWrightHighRank}.  Since $\bfC_1$ is involved with rel, $\cM_w$ is rank two rel zero (by \cite[Lemma 11.4]{ApisaWrightHighRank}). If $\cM_w$ consists of connected surfaces, then we are done. If not, then, by Lemma \ref{lemma:ConnCompsDeg}, $\cM_w$, after forgetting marked points, is prime and contained in $\cH(2) \times \cH(2)$. On each component, the cylinders in $\Col_w(\bfC_1)$ consist of one of the following: (1) one simple cylinder, (2) two simple cylinders formed by marking a Weierstrass point in a simple cylinder, or (3) two half-simple cylinders formed by marking two Weierstrass points in a simple cylinder. Let $\bfC_1'$ be the corresponding cylinders in $\bfC_1$, i.e. the cylinders in $\bfC_1$ that remain positive area when passing to $\Col_w(X, \omega)$. Let $v$ be a typical degeneration in the twist space of $\bfC_1$ so that $(\bfC_1)_v \subseteq \bfC_1'$. (To see that this exists, let $v$ be the pullback of the standard shear on $\Col_w(\bfC_1)$ to the cylinders in $\bfC_1'$. This can be written as $\sum_{c \in \bfC_1'} h_c \gamma_c^*$ where $h_c$ is positive real and $\gamma_c^*$ is the Poincar\'e dual of the cylinder core curve of $c \in \bfC_1'$. Perturb $(X, \omega)$ so that any two cylinders $c$ and $c'$ in $\bfC_1'$ have a ratio of heights equal to $h_{c}/h_{c'}$ only if this is generically true. Then $v$ is typical.) The explicit description of the cylinders in $\bfC_1'$ shows that no matter how many cylinders in $\bfC_1'$ we collapse, the remaining cylinders in $\bfC_1'$ are still simple or half-simple. Since collapsing a simple or half-simple cylinder preserves connectedness, $\Col_v(X, \omega)$ is connected as desired.
\end{proof}

%
%
%
%
%

\begin{sublemma}\label{SL:ConnectedDoubleDeg}
After perhaps choosing a different typical degeneration $v$, $\Col_{v, \bfC_2}(X, \omega)$ is connected.
\end{sublemma}
\begin{proof}
Suppose not.  Then, by Lemma \ref{L:DisconnectedRank2Rel0Case}, $\Col_v(\bfC_2)$ is in the golden configuration and $\Col_v(\bfC_1)$ is a pair of simple cylinders. Let $\bfC_1'$ be the corresponding cylinders in $\bfC_1$. Let $v'$ be a typical degeneration so that $(\bfC_1)_{v'} \subseteq \bfC_1'$. Since the collapsing cylinders are simple, the surfaces in $\cM_{v'}$ are connected and, as with $\cM_v$, $\cM_{v'}$ must be the $(1,1,1,7)$-locus. If $\Col_{v', \bfC_2}(X, \omega)$ is connected, we are done. So suppose that this does not happen and so, again by Lemma \ref{L:DisconnectedRank2Rel0Case}, $\Col_{v'}(\bfC_1)$ is a pair of simple cylinders.  Therefore, $\bfC_1$ consists of three or four simple cylinders. Since $\cM$ has no marked points, no two of these cylinders share boundary saddle connections. Therefore, $\Col_{\Col_v(\bfC_1)} \Col_{v}(X, \omega)$ may be identified with the surface shown in Figure \ref{F:example:sub1} with $\Col_{\Col_v(\bfC_1)} \Col_{v}(\bfC_2)$ the two horizontal cylinders in the golden configuration. But there are only two vertical saddle connections in the complement of these cylinders,  contradicting that $\Col_{\Col_v(\bfC_1)} \Col_{v}(\bfC_1)$ consists of either three or four vertical saddle connections.
\end{proof}

By Apisa-Wright \cite[Lemma 3.8]{ApisaWrightHighRank}, $\cM_{\bfC_2}$ is rank one and, since $\bfC_2$ is generic, necessarily rank one rel two. By Sublemma \ref{SL:ConnectedDoubleDeg}, $\cM_{\bfC_2}$ consists of connected genus two surfaces. By Apisa-Wright \cite[Lemma 10.5]{ApisaWrightHighRank}, which applies by Lemma~\ref{lemma:NonArithConnProdTwoSC}, $\cM_{\bfC_2}$ has no free marked points. (Although the cited result is stated for rel zero invariant subvarieties, its conclusion and proof hold verbatim in our setting of a rank reducing deformation of a generic equivalence class of cylinders with one-dimensional twist space.) We also have that $\cM_{v,\bfC_2}$ is the golden eigenform locus with the golden point marked. Therefore, $\cM_{\bfC_2}$ consists of surfaces in the golden eigenform locus with the golden point $q$ marked and a pair of marked points $\{p,p'\}$ exchanged by the hyperelliptic involution (see Apisa-Wright \cite{ApisaWrightMPts}). Since $\cM$ has no marked points, $\Col_{\bfC_2}(\bfC_2)$ consists of generically parallel saddle connections with at least one endpoint belonging to $\{p, p'\}$ and with at least one saddle connection ending at $q$. Therefore, $\Col_{\bfC_2}(\bfC_2)$ consists of two saddle connections, one, call it $s_1$, joining $q$ to $p$ and one, call it $s_2$, joining $p'$ to one of the following $p$, $q$, or a zero. Send the length of $s_1$ to zero by moving $p$ along it in the direction of $q$. It follows that when $p$ and $q$ coincide, $p'$ must be at the image of $q$ under the hyperelliptic involution and so this must be an endpoint of $s_2$. But the image of $q$ under the hyperelliptic involution is not a zero or $q$, a contradiction.
\end{proof}

\begin{proof}[Proof of Theorem \ref{T:Main-tool2}:] We will derive a contradiction.  By Lemma~\ref{L:GenEqImpNoRel}, no equivalence class has rel in its twist space. Let $\bfC_1$, $\bfC_2$, and $\bfC_3$ be three disjoint generic equivalence classes of cylinders on a surface $(X, \omega) \in \cM$. These exist by perturbing a cylindrically stable surface created in Apisa-Wright \cite[Lemma 7.10]{ApisaWrightHighRank}. Necessarily, all equivalence classes are involved in rel.  If only two were, then any two cylinders in those equivalence classes would be generically parallel and hence equivalent. Therefore, $\cM_{\bfC_i}$ is rank two rel zero for all $i$. By Proposition \ref{P:induction}, without loss of generality, $\cM_{\bfC_3}$ and $\cM_{\bfC_1}$ consist of connected surfaces.  Hence, $\cM_{\bfC_3}$ and $\cM_{\bfC_1}$ are both the $(1,1,1,7)$-locus by Theorem \ref{T:main-tool}.  Since the generic cylinder equivalence classes in the $(1,1,1,7)$ locus are pairs of simple cylinders or a golden configuration, the same is true for equivalence classes in $\cM$. 

\begin{sublem}\label{SL:AllSimpleAgain}
It is not the case that $\bfC_i$ is a pair of simple cylinders for each $i \in \{1,2,3\}$.
\end{sublem}
\begin{proof}
Suppose not. Then $\cM_{\bfC_1}$ is the $(1,1,1,7)$-locus, $\cM_{\bfC_1, \bfC_2}$ is the golden eigenform locus in $\cH(1,1,0)$ with the golden point marked, and $\cM_{\bfC_1, \bfC_2, \bfC_3}$ is the golden eigenform locus in $\cH(2)$. On this surface, the cylinders in $\bfC_i$, for each $i \in \{1,2,3\}$ have become generically parallel saddle connections. Since there is one cusp on this Teichm\"uller curve, we may suppose that $\Col_{\bfC_1, \bfC_2, \bfC_3}(X, \omega)$ is horizontally periodic and that $\Col_{\bfC_1, \bfC_2, \bfC_3}(\bfC_1)$ consists of two horizontal saddle connections.  Therefore, the surface $\Col_{\bfC_2, \bfC_3}(X, \omega)$ is the one shown in Figure \ref{F:example:sub1}. As argued in Sublemma \ref{SL:GoldenSC}, the saddle connections in $\Col_{\bfC_2, \bfC_3}(\bfC_2)$ and $\Col_{\bfC_2, \bfC_3}(\bfC_3)$, up to Dehn twists, are the green and purple ones in Figure \ref{F:example2:sub1}\footnote{The purple saddle connections in Figure~\ref{F:example2:sub1} are the parallel pair with the smaller negative slope, and the green saddle connections are those with the larger negative slope.}. To rule this out, consider $\Col_{\bfC_1}(X, \omega)$, i.e. collapse the pair of cylinders labeled $\Col_{\bfC_1}(\bfC_2)$ in Figure \ref{F:example2:sub1} and glue in the two pairs of simple cylinders to the green and purple saddle connections. Up to shearing this is the configuration ruled out in Sublemma \ref{SL:RulingOutFigure5}. This is justified using the cut and paste argument in Figure \ref{fig:four-slits-to-six-a}.
\begin{figure}[htbp]
\centering
\begin{tikzpicture}[x=1.20cm,y=1.20cm,line cap=round,line join=round,
  edge/.style={black!65,thin},
  old/.style={black!55,densely dashed,thin},
  G/.style={slitgreen,line width=.9pt},
  P/.style={slitpurple,line width=.9pt},
  A/.style={brown!85!black,line width=.9pt},
  B/.style={brown!85!black,line width=.9pt,
            dash pattern=on 3pt off 1.5pt},
  H/.style={teal!80!black,line width=1pt,densely dotted},
  dot/.style={circle,draw=black,fill=paperblue,inner sep=0pt,
              minimum size=2.8pt,line width=.3pt},
  num/.style={font=\fontsize{6.5}{7.5}\selectfont,inner sep=0pt},
  ttl/.style={font=\small,inner sep=1pt},
  lab/.style={font=\scriptsize,inner sep=1.4pt}]

\definecolor{slitgreen}{HTML}{2E7D32}
\definecolor{slitpurple}{HTML}{6557D2}
\definecolor{paperblue}{HTML}{1565C0}
\definecolor{paperred}{HTML}{D32F2F}
\pgfmathsetmacro{\K}{(1+sqrt(5))/2}
\pgfmathsetmacro{\Lside}{1+\K}
\pgfmathsetmacro{\q}{\K-1}

\newcommand{\dotsat}[1]{%
  \foreach \x/\y in {#1}{\node[dot] at (\x,\y) {};}}

\newcommand{\pieceof}[1]{%
  \coordinate (A) at (0,\K);       \coordinate (B) at (1,\K);
  \coordinate (C) at (\Lside,\K);  \coordinate (D) at (0,0);
  \coordinate (E) at (1,0);        \coordinate (F) at (\Lside,0);
  \coordinate (H) at (1,-1);       \coordinate (I) at (\Lside,-1);
  \coordinate (U) at (\K,0);       \coordinate (V) at (2,0);
  \ifcase#1\relax
  \or 
    \draw[old] (D)--(A);
    \draw[A] (D)--(E); \draw[G] (A)--(E);
    \node[num] at (.42,.51) {$1$};
  \or 
    \draw[G] (A)--(E); \draw[old] (H)--(E);
    \draw[B] (H)--(I); \draw[P] (A)--(I); \draw[H] (E)--(U);
    \node[num] at (.958,.408) {$2^+$};
    \node[num] at (1.81,-.50) {$2^-$};
  \or 
    \draw[A] (A)--(B); \draw[G] (B)--(I);
    \draw[P] (A)--(I); \draw[H] (U)--(V);
    \node[num] at (1.07,.96) {$3^+$};
    \node[num] at (2.026,-.252) {$3^-$};
  \or 
    \draw[P] (B)--(F); \draw[old] (I)--(F);
    \draw[G] (B)--(I); \draw[H] (V)--(F);
    \node[num] at (1.958,.408) {$4^+$};
    \node[num] at (2.356,-.314) {$4^-$};
  \or 
    \draw[B] (B)--(C); \draw[old] (F)--(C); \draw[P] (B)--(F);
    \node[num] at (2.11,1.13) {$5$};
  \fi}

\begin{scope}
  \foreach \i in {1,...,5}{\pieceof{\i}}
  \dotsat{0/\K,1/\K,\Lside/\K,0/0,1/0,\Lside/0,1/-1,\Lside/-1}
  \node[ttl] at ({\Lside/2},-1.35) {(a) The slits of Figure~4a};
\end{scope}

\begin{scope}[xshift=5.55cm]
  \fill[black!6] (\K,0) rectangle (\Lside,\K);
  \draw[edge] (\K,0)--(\K,\K);
  \foreach \i/\sx/\sy/\tx/\ty in
    {1/0/0/\K/\K,2/1/0/\K/0,3/\Lside/-1/0/\K,
     4/\Lside/0/\K/0,5/\Lside/0/\K/\K}{
    \begin{scope}[shift={(\tx,\ty)},cm={-\K,-\K,-1,-\K,(0,0)},
                  shift={({-\sx},{-\sy})}]
      \pieceof{\i}
    \end{scope}}
  \dotsat{0/-1,0/0,0/\K,\K/\K,\Lside/\K,\Lside/0,\K/0,\K/-1}
  \node[ttl] at ({\Lside/2},-1.35) {(b) Reglue along the old sides};
\end{scope}

\newcommand{\horizontaltraces}{%
  \draw[A] (\J,0)--(\R,\K);
  \draw[B] (\J,-1)--(\R,\q) (0,\q)--(1,\K);
  \draw[H] (\R,0)--({\J+\q},-1)
           ({\J+\q},\K)--(\J,1) (1,1)--(0,0);}

\begin{scope}[yshift=-4.85cm]
  \def\J{1}\let\R\Lside
  \fill[black!6] (0,0) rectangle (1,\K);
  \draw[edge] (0,0)--(0,\K) (\R,0)--(\R,\K);
  \draw[old] (1,0)--(\R,0);
  \draw[G] (1,-1)--(1,\K) (\R,-1)--(\R,0);
  \draw[P] (0,\K)--(\R,\K) (0,0)--(1,0) (1,-1)--(\R,-1);
  \horizontaltraces
  \dotsat{0/\K,1/\K,\R/\K,0/0,1/0,\R/0,1/-1,\R/-1}
  \node[ttl] at ({\R/2},-1.88) {(c) Move the shaded strip to the left};
\end{scope}

\begin{scope}[xshift=5.25cm,yshift=-4.85cm,x=.96cm,y=.96cm]
  \def\hs{.43}\def\as{.52}
  \pgfmathsetmacro{\J}{1+\K*\hs}
  \pgfmathsetmacro{\R}{\J+\K}
  \pgfmathsetmacro{\S}{\R+\hs}
  \pgfmathsetmacro{\T}{\K+\as}
  \pgfmathsetmacro{\botY}{-1-\K*\as}

  \draw[edge] (0,0)--(0,\K)--(\R,\K)--(\R,-1)--(\J,-1)--(\J,0)--cycle;
  \draw[paperred,fill=paperred!10] (1,0) rectangle (\J,\K);
  \draw[paperred,fill=paperred!10] (\R,-1) rectangle (\S,0);
  \draw[paperblue,fill=paperblue!10] (0,\K) rectangle (1,\T);
  \draw[paperblue,fill=paperblue!10] (\J,\botY) rectangle (\R,-1);
  \horizontaltraces

  \node[lab] at ({(1+\J)/2},{\K/2}) {$\mathbf C_1$};
  \node[lab] at ({(\R+\S)/2},-.5) {$\mathbf C_1$};
  \node[lab] at (.5,{(\K+\T)/2}) {$\mathbf C_2$};
  \node[lab] at ({(\J+\R)/2},{(\botY-1)/2}) {$\mathbf C_2$};

  \foreach \x/\y/\t in {0/{(\K+\T)/2}/a,0/{\K/2}/b,
                         \J/-.5/c,\J/{(\botY-1)/2}/d}
    \node[lab,left] at ({\x},{\y}) {$\t$};
  \foreach \x/\y/\t in {.5/\T/e,{(1+\J)/2}/\K/f,
                         {(\J+\R)/2}/\K/g,{(\R+\S)/2}/0/h}
    \node[lab,above] at ({\x},{\y}) {$\t$};

  \dotsat{0/\T,1/\T,0/\K,1/\K,\J/\K,\R/\K,0/0,1/0,\J/0,\R/0,\S/0,
          \J/-1,\R/-1,\S/-1,\J/\botY,\R/\botY}
  \node[ttl] at ({\S/2},-2.35) {(d) The presentation of Figure~6a};
\end{scope}
\end{tikzpicture}
\caption{On the golden eigenform $\Col_{\bfC_1, \bfC_2, \bfC_3}(X, \omega) \in \cH(2)$ we have three pairs of saddle connections $(S_i)_{i=1}^3$, where $S_i := \Col_{\bfC_1, \bfC_2, \bfC_3}(\bfC_i)$. These are the solid and dashed brown horizontal saddle connections in $S_1$ and the green and purple (solid diagonal) saddle connections in $S_2$ and $S_3$. We have labeled the complement of these saddle connections by $1/5, 2/4, 3$. Adding in the two dashed uncolored vertical saddle connections subdivides the components into those labeled $1, \hdots, 5$. Adding in the dotted teal saddle connection further subdivides these pieces into the parts above ($+$) and below ($-$) it in (a). The same labels follow these parts in (b). To move from (a) to (b) we apply an element of $\mathrm{GL}(2, \mathbb{R})$ that makes the green saddle connections vertical and the purple ones horizontal. In (c) the shaded strip has moved to the left end of the same cylinder. By gluing in the pair of simple cylinders in $\bfC_2$ and $\bfC_3$ to the green and purple saddle connections in (c), we arrive at the surface in (d), which was shown not to have rank two rel zero orbit closure in Sublemma \ref{SL:RulingOutFigure5}.}
\label{fig:four-slits-to-six-a}
\end{figure}
\end{proof}

Suppose therefore that $\bfC_2$ consists of cylinders in the golden configuration. By assumption, $\cM_{\bfC_i}$ is the $(1,1,1,7)$-locus when $i = 1, 3$. By Lemma \ref{lemma:genusbytwo}, $\cM_{\bfC_1, \bfC_3}$ consists of either connected genus two surfaces or a pair of disconnected genus one surfaces. Since $\bfC_2$ is the golden configuration, it remains so on $\Col_{\bfC_1, \bfC_2}(X, \omega)$ and so this surface is connected and genus two. By Proposition \ref{P:GenusTwoCylinderRigid}, $\Col_{\bfC_1, \bfC_3}(X, \omega)$ is the golden eigenform locus with the golden point marked. 

As argued\footnote{The argument ran as follows. If not, then, up to relabeling, $\Col_{\bfC_1}(X, \omega)$ would be the surface depicted in Figure \ref{F:example:sub2}. This surface has an equivalence class of cylinders shown in Figure \ref{F:example2:sub2}, which, when collapsed, does not result in a surface in $\cH(1,1,0)$, contradicting Proposition \ref{P:GenusTwoCylinderRigid}.} in Sublemma \ref{SL:AllSimple}, $\bfC_1$ and $\bfC_3$ must be pairs of simple cylinders. By Proposition \ref{P:GenusTwoCylinderRigid}, $\Col_{\bfC_1, \bfC_3}(X, \omega)$ may be taken to be the horizontally periodic surface in Figure \ref{F:example}, with the cylinders in $\Col_{\bfC_1, \bfC_3}(\bfC_2)$ corresponding to those in $\bfB$ in the figure. Shearing so that $\Col_{\bfC_1, \bfC_3}(\bfC_3)$ is a pair of vertical saddle connections in its complement, it follows that $(X, \omega)$ is the surface shown in Figure \ref{F:Example5} or its mirror image about the vertical axis.

Suppose for the sake of concreteness that all six shaded cylinders are squares and that the larger red, white, and blue squares have side lengths $1$, $1$, and $\phi$. Then one can check that the direction with slope $\sqrt{5}$ decomposes the surface into three cylinders, contradicting the claim that every equivalence class has exactly two cylinders (as noted before Sublemma \ref{SL:AllSimpleAgain}).  A SageMath Notebook verifying the claim can be found at \cite{aulicino_papers}.
\end{proof}

\begin{figure}[H]
 \centering
  \includegraphics[width=.75\linewidth]{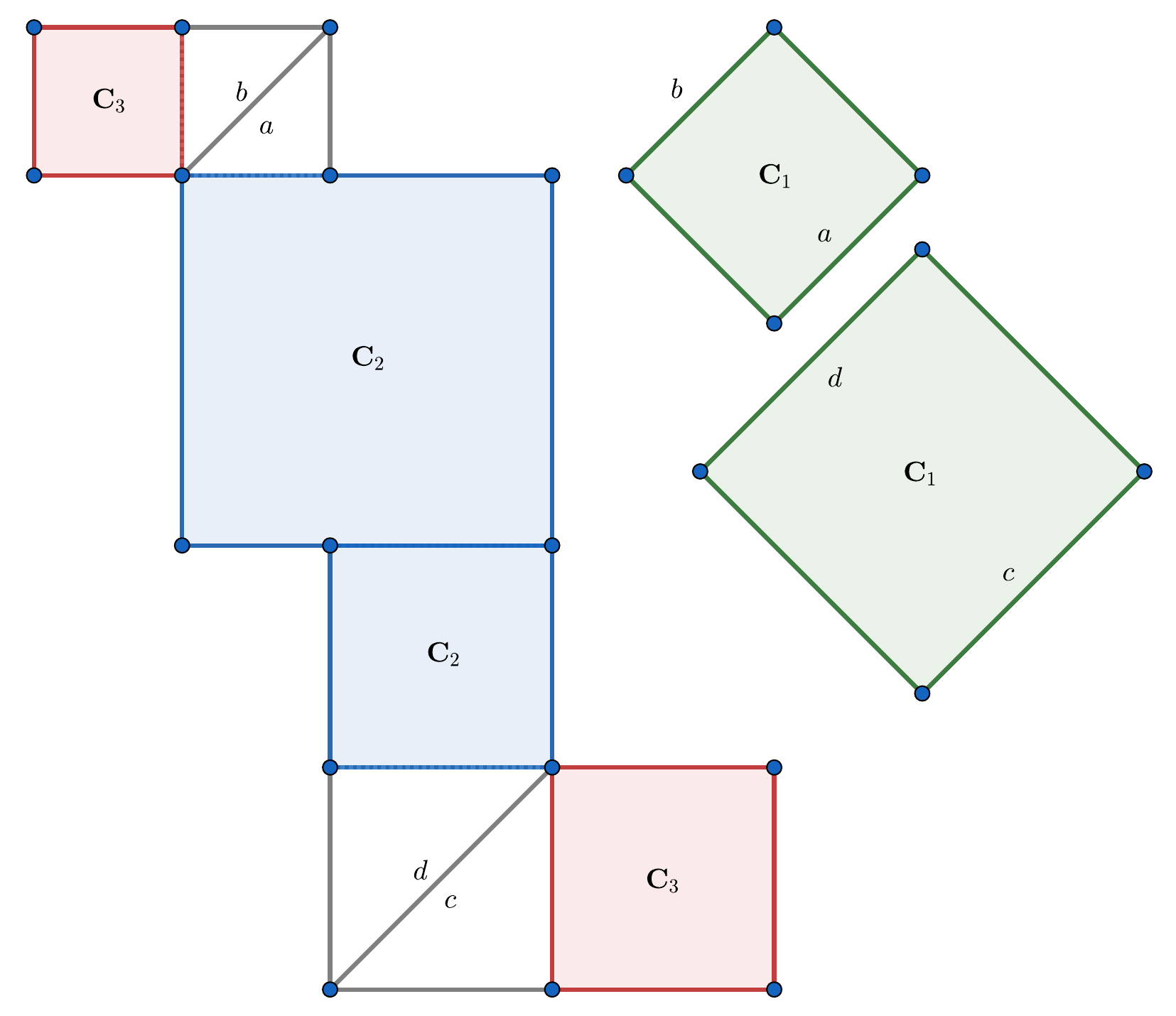}
\caption{The surface in the proof of Theorem \ref{T:Main-tool2}. Opposite sides of the polygon are identified unless otherwise specified.}
\label{F:Example5}
\end{figure}

\bibliography{fullbibliotex}{}
\end{document}